\documentclass{daj}

\usepackage{macros-daj}

\dajAUTHORdetails{%
  title = {Products of Differences over Arbitrary Finite Fields}, 
  author = {Brendan Murphy and Giorgis Petridis},
  plaintextauthor = {Brendan Murphy, Giorgis Petridis},
  keywords = {sum product problem, finite fields},
}   

\dajEDITORdetails{%
   year={2018},
   number={18},
   received={6 July 2017},   
   published={7 November 2018},  
   doi={10.19086/da.5098},       
}   

\begin{document}

\begin{frontmatter}[classification=text]


\author[bm]{Brendan Murphy\thanks{Supported by the Heilbronn Institute for Mathematical Research}}
\author[gp]{Giorgis Petridis\thanks{Supported by the NSF DMS Award 1723016 and gratefully acknowledges funding support from the RTG in Algebraic Geometry, Algebra, and Number Theory at the University of Georgia, and from the NSF RTG grant DMS-1344994.}}

\begin{abstract}
There exists an absolute constant $\delta > 0$ such that for all $q$ and all subsets $A \subseteq \mathbb{F}_q$ of the finite field with $q$ elements, if $|A| > q^{2/3 - \delta}$, then
\[
|(A-A)(A-A)| = |\{ (a -b) (c-d) : a,b,c,d \in A\}| > \frac{q}{2}.
\]
Any $\delta < 1/13,542$ suffices for sufficiently large $q$. This improves the condition $|A| > q^{2/3}$, due to Bennett, Hart, Iosevich, Pakianathan, and Rudnev, that is typical for such questions.

Our proof is based on a qualitatively optimal characterisation of sets $A,X \subseteq  \mathbb{F}_q$ for which the number of solutions to the equation
\[
(a_1-a_2) = x (a_3-a_4) \, , \; a_1,a_2, a_3, a_4 \in A, x \in X
\]
is nearly maximum.
A key ingredient is determining exact algebraic structure of sets $A, X$ for which $|A + XA|$ is nearly minimum, which refines a result of Bourgain and Glibichuk using work of Gill, Helfgott, and Tao.

We also prove a stronger statement for 
\[
(A-B)(C-D) = \{ (a -b) (c-d) : a \in A, b \in B, c \in C, d \in D\}
\]
when $A,B,C,D$ are sets in a prime field, generalising a result of Roche-Newton, Rudnev, Shkredov, and the authors.
\end{abstract}
\end{frontmatter}

\section{Introduction}
\label{sec:intro}

\subsection{Overview}

A manifestation of the sum-product phenomenon is expansion of polynomial images. The principle asserts that if a polynomial does not have a 
specific form, then it takes many  distinct values on large finite Cartesian products.

The first such result was proved by Elekes over the real numbers for two-variable polynomials~\cite{Elekes1998}. Elekes and R\'onyai proved a stronger result for rational functions~\cite{Elekes-Ronyai2000}. Their result, when restricted to polynomials, roughly asserts that there exist absolute constants $c$ and $\delta$ such that if $f(x,y) \in \R[x,y]$ is not similar to $x+y$ or $xy$, then for all finite sets $A \subset \R$ we have $|f(A,A)| > c |A|^{1+\delta}$. One cannot hope for such expansion for, say, $f(x,y) = x+3y$ as we see by letting $A$ be an arithmetic progression. Strong qualitative bounds were given by Raz, Sharir, and Solymosi, who showed that one may take $\delta = 1/3$~\cite{RSS2016}. 

Better quantitative bounds are known for what one could call
natural expander polynomials
: $x+yz$~\cite{Shkredov2016B}, $x(y+z)$~\cite{MRS}, $wx+yz$~\cite{IRNR}, $(w-x)^2+(y-z)^2$~\cite{Guth-Katz2015}, and $(w-x)(y-z)$~\cite{RNR2015}.

Several formulations of polynomial expansion have also been studied over finite fields. 

We study a variation of what Tao calls (symmetric) moderate expanders~\cite{Tao2015},
focusing on quadratic polynomials with coefficients in $\{0, \pm 1\}$. 
The precise expansion property we study is as follows:
Given a polynomial $f \in \mathbb{Z}[X_1,\dots, X_n]$ we seek absolute constants $c,C$, and $\alpha$ such that for all prime powers $q$ and for all sets $A \subseteq \F_q$ that satisfy $|A| > C q^{\alpha}$ we have
\[
|f(A,\dots,A)| > cq.  
\] 
Tao established an analogue of the Elekes--R\'onyai theorem for two-variable polynomials with $\alpha = 15/16$~\cite{Tao2015}.
Earlier, Hart, Li, and Shen proved stronger expansion for a smaller class of polynomials with the smaller threshold $\alpha=3/4$~\cite[Theorem 2.11]{HLS2013}.
Expansion conditional on $A$ having few distinct sums was obtained even earlier by Solymosi~\cite[Corollary 4.4]{Solymosi2008}.
See also the papers~\cite{Bukh-Tsimerman2012,Gyarmati-Sarkozy2008,Hegyvari-Hennecart2013,Shkredov2010,Vu2008}. 

Expansion has been established for many natural examples of expanding polynomials with the smaller common threshold $\alpha = 2/3$:
\begin{itemize}
\item  $x+yz$~\cite[Inequality (9)]{Shparlinski2008} or~\cite[Theorem 3.2]{Vinh2013}, 
\item $x(y+z)$~\cite{GPInc},
\item  $wx+ y + z$~\cite{Sarkozy2005}, 
\item $wx+yz$~\cite{Hart-Iosevich2008}, 
\item $(w-x)^2+(y-z)^2$~\cite{CEHIK2012} (here we must require $q$ to be odd), and
\item  $(w-x)(y-z)$~\cite[Corollary~1.8]{BHIPR2017}. 
\end{itemize}

Using standard notation for set addition and multiplication, one may state the last result for $(w-x) (y-z)$ as: If $|A| > q^{2/3}$, then
\[
|\pda| = |\{ (a -b) (c-d) : a,b,c,d \in A\}| > \frac{q}{2}.
\]

All of these results have straightforward proofs 
using the \emph{completion method}, which roughly involves replacing a sum of squares over a large set 
in a space by the sum of squares over the whole space.
This seemingly expensive step allows one to take advantage of properties like character orthogonality and produces non-trivial results for large sets. 
The completion method has been applied to additive characters, multiplicative characters, and the spectral theory of regular graphs. 

A unifying approach for 
 ``pinned'' versions of the results was presented in~\cite{GPInc}.
The fact that two points in the two-dimensional vector space $\F_q^2$ determine a unique line, combined with a second moment argument, is used as a substitute for the so-called expander mixing lemma~\cite{Haemers1979,Alon-Chung1988}.
The corresponding result for the polynomial $(w-x)(y-z)$ was omitted from~\cite{GPInc}, but it is straightforward to prove that if $|A| > q^{2/3}$, then there exist $a, b \in A$ such that
\[
|(A-a)(A-b)| = |\{ (c -a) (d-b) : c,d \in A\}| > \frac{q}{2}.
\]

\subsection{Main result}

Going below the $q^{2/3}$ threshold for any of the polynomials listed above has proven to be surprisingly hard, even in prime order fields where there are additional tools. Let us write $p$ for the order of the field when it is a prime. When $A \subseteq \F_p$ is a multiplicative subgroup, Heath-Brown and Konyagin applied Stepanov's method~\cite{Heath-Brown-Konyagin2000} to show that $|A+A| > c \min\{p, |A|^{3/2}\}$, matching the $p^{2/3}$ threshold  (note that $A+A=A+AA = A(A+A)$).
Similarly, Pham, Vinh, and de Zeeuw used Rudnev's point-plane incidence bound in $\F_p^3$~\cite{Rudnev,dZ} to prove that $|f(A,A,A)| > c \min\{p , |A|^{3/2}\}$ for a large class of quadratic polynomials in three variables~\cite{PVZ}.

The first step in going below the $p^{2/3}$ threshold for prime $p$ was taken in~\cite{GPProdDiff} for $(w-x)(y-z)$, where it was shown that there exists a constant $C$ such that if $|A| > C p^{5/8}$, then $| \pda | >  p/2$.
This argument does not carry through to arbitrary $\F_q$, because it relies on non-trivial bounds on the number of ordered solutions to the equation $ (a_1-a_2) = x (a_3-a_4)$ with the $a_i \in A$ and $x$ in an arbitrary set $X \subseteq \F_p$.
Bourgain~\cite[Theorem C]{Bourgain2009} proved a non-trivial bound in $\F_p$, but such a theorem is false in $\F_q$,  in general. 

Roche-Newton, Rudnev, Shkredov, and the authors~\cite[Theorem 27]{GPSumProd} improved this result, showing that $|A| > C p^{3/5}$ implies the ``pinned'' result $|(A-a)(A-b)|> p/2$ for some $a,b \in A$.
The proof uses a bound on the number of ordered collinear triples in $A \times A \subseteq \F_p^2$ that is better than what follows from the completion method. In terms of incidence geometry results, the method of proof combines Vinh's point-line incidence theorem~\cite{Vinh2011} with a point-line incidence theorem of Stevens and de Zeeuw~\cite{SdZ}.  This method is also specific to prime fields and does not carry over in arbitrary $\F_q$ either.

Our main theorem breaks the $q^{2/3}$ threshold over arbitrary finite fields. 
\begin{thmalph}
  \label{thm:A}
There exist absolute positive constants $C,k > 0$ such that for all prime powers $q$, if $A \subseteq \F_q$ satisfies
\[
|A| > C q^{\tfrac 2 3 - \tfrac 1 {13,542}} \log^k(q),
\]
then $|\pda| > q/2$.
\end{thmalph}
The proof of Theorem~\ref{thm:A} uses
most of the existing techniques for sum-product questions in arbitrary finite fields at different stages of the proof: the completion method in Section~\ref{sec:thmB}, the pivot method in Section~\ref{sec:thmC}, and character sum estimates in Section~\ref{sec:thmD}.

Before discussing the proof of Theorem~\ref{thm:A} in more detail, we make some remarks. 

\begin{remarks}
\begin{enumerate}
\item Over  prime finite fields $\F_p$ we prove a stronger result for $(A-B)(C-D)$.
Using a result in~\cite{GPSumProd}, we show that if $(|A||B||C||D|)^{1/4} > p^{3/5}$, then $|(A-B)(C-D)| > cp$.
This proves expansion for rational functions such as $(w+x^2)(y-z^{-1})$ with $|A|$ as small as $p^{3/5}$.  

\item The proof of Theorem~\ref{thm:A} contains a simple example of a set $A\subset \F_q$ of cardinality $|A| = q^{2/3}$ with the property $|\pda | = (1+o(1)) q/2$; see Proposition~\ref{VV}.
This shows that $\pda$ cannot have density greater than 1/2, in general.
Hart, Iosevich, and Solymosi~\cite{HIS2007} conjectured that there exists $\varepsilon >0$ such that if $|A| > C q^{1/2 + \varepsilon}$, then $\pda = \F_q$.
The example shows that $\varepsilon$ must be greater than $1/6$. The best known lower bound on $|A|$ for this complementary question is $C q^{3/4}$ and is due to Hart, Iosevich, and Solymosi~\cite[Theorem 1.4]{HIS2007}. 

\item Vinh has obtained expansion with threshold $q^{5/8}$ for the four-variable polynomial $ wx + (y-z)^2$~\cite{Vinh2013}
in finite fields $\F_q$ with odd characteristic.
Vinh's elegant proof is based on ideas of Garaev~\cite{Garaev2008} and Solymosi~\cite{Solymosi2008} and strongly relies on having a polynomial where the additive and multiplicative parts are separated.
 
\item Rudnev, Stevens, and Shkredov have proved the strongest expansion result in $\F_p$ we are aware of. They worked with the four-variable rational function $(wx-y) (w-z)^{-1}$~\cite{RSS} and showed there exist absolute constants $c, C$, and $k$ such that if $A \subseteq \F_p$ satisfies $|A| > C \log(p)^{k} p^{25/42}$, then $|f(A,A,A,A)| \geq c p$. The exponent $25/42$ is slightly smaller than the exponent $3/5$, which is known for $(w-x)(y-z)$ ~\cite{GPSumProd}, the three variable rational function $(x-y)(x-z)^{-1}$~\cite{GPSumProd}, and the four variable rational functions covered by the forthcoming Theorem~\ref{thm:E}.
\end{enumerate}
 \end{remarks}

\subsection{Proof outline and structure of the paper}

Theorem~\ref{thm:A} is proved using a dichotomy argument, carried out in three steps: a completion argument that works for generic sets, an exact structural characterisation of sets for which the completion argument fails, and a character sum  argument that deals with 
the structured sets. 

Standard results from arithmetic combinatorics used throughout the paper are given in Section~\ref{sec:Prelim}.
To make the statements of the forthcoming theorems more compact, we make one preliminary definition: for $A \subseteq \F_q$ and $\xi \in \F_q^*$ we use $E(A, \xi A)$ to denote the number of ordered solutions to the equation
\[
 (a_1-a_2) =\xi (a_3-a_4)
\] 
with all variables $a_1,\ldots,a_4$ in $A$.

\subsubsection*{Section~\ref{sec:thmB}: The generic sets completion argument}

The first step is similar to arguments used to deal with large subsets of finite fields.
It mimics the argument in~\cite{GPProdDiff} to show that Theorem~\ref{thm:A} holds unless there exists a set $X \subset \F_q$  of cardinality roughly $q/|A|$ 
such that $E(A , \xi A)$ is greater than $|A|^3 / K$, for all $\xi \in X$ and some suitable $K$. This is is achieved by a careful analysis of a completion argument.

Define the quantity
\begin{equation}
   \label{eq: D(A)}
D_\times(A) = | \{ ( a_1, \dots, a_8)\in A^8 : (a_1-a_2) (a_3-a_4) = (a_5-a_6) (a_7-a_8) \}|.
\end{equation}
By Cauchy-Schwarz, $|\pda|\cdot D_\times(A)\geq |A|^8$.
We aim to show that $D_\times(A)$ is at most $2 |A|^8/q$, under suitable conditions.

More precisely, we prove the following theorem in Section~\ref{sec:thmB}.
\begin{thmalph}
  \label{thm:B}
Let $A$ be a subset of $\F_q$.
Suppose that there is a positive constant $K$ such that $|A| \leq  q/(4K)$ and
\begin{equation*}
 \label{eq:23}
  D_\times(A) \geq \frac{|A|^8}q + \frac{3q|A|^5}K.
\end{equation*}
There exists a subset $X\subseteq\F_q$ such that $E(A,\xi A) \geq |A|^3/K$ for all $\xi \in X$ and
\[
\frac{q}{K|A|}\leq |X|\leq \frac{4K q}{3|A|}.
\]
\end{thmalph}

It is easy to show, for example using multiplicative characters, that $\ds D_\times(A) = |A|^8/q + O(q |A|^5)$, yielding the $|A| > C q^{2/3}$ threshold. The $|A|^8/q$ term, which comes from the principal multiplicative character, is the expected value of $D_\times(A)$ for large random sets.

\subsubsection*{Section~\ref{sec:thmC}: Inverse sum-product results}

The next step in the proof  is to show that the conclusion of Theorem~\ref{thm:B} is impossible for $K=q^\delta$ for some small $\delta>0$, unless $A$ has a specific structure.

A result of Bourgain implies that the conclusion in Theorem~\ref{thm:B} is never true in $\F_p$ when $K$ is small power of $p$ and $|A| |X|$ is about $p$. 
Namely, \cite[Theorem C]{Bourgain2009} states that there exists an absolute constant $\eta>0$ such that for sets $A, X \subseteq \F_p$ satisfying  $|A| |X| < p$, we have
\[
\sum_{\xi \in X} E(A, \xi A) \leq |A|^3 |X|^{1-\eta}.
\]
This suffices to go below $p^{2/3}$ in $\F_p$~\cite{GPCANTExp,GPProdDiff}.

However, Bourgain's result is not true over arbitrary finite fields. If $X$ is contained in a proper subfield and $A$ is a vector space over the subfield, then
$E(A, \xi A) = |A|^3$ for all $\xi \in X$.
Taking $q = p^3$, $X$ to be the subfield isomorphic to $\F_p$ and $A$ to be a two-dimensional vector space over $X$, produces a pair of sets $A$ and $X$ with $|A| |X| = q$, where the conclusion (and therefore the hypothesis) of Theorem~\ref{thm:B} is false.

Therefore, we must deal with pairs of sets $A, X$ where 
\begin{equation}
  \label{eq:hypothesis bad}
\sum_{\xi \in X} E(A, \xi A) \geq \frac{|A|^3 |X|}{K}
\end{equation}  
for a parameter $K$ that will eventually be taken to be a small power of $q$.

Bourgain and Glibichuk proved in~\cite[Proposition~2]{Bourgain-Glibichuk2011} that if~\eqref{eq:hypothesis bad} holds, then $X$ must have large intersection with the dilate of a subfield.
This characterisation is hard to apply in our context because we are more interested in $A$ than in $X$.
We strengthen this result to show that the example given above is essentially the only way that~\eqref{eq:hypothesis bad} can hold: a large subset of $X$ must be contained in a dilate of a subfield and $A$ must have a large intersection with a vector space over the same subfield.
Though $X$ itself may be small compared to the subfield, $A$ must have comparable size to the vector space.
This characterisation is qualitatively best possible.

The next theorem contains a precise statement of this result.
For brevity, we write $ f \lesssim g$ if there exist constants $C, k$ such that $f \leq C \log^k(q) \, g$.
\begin{thmalph}
  \label{thm:C}
Let $A$ and $X$ be subsets of $\F_q^*$ and let $K\geq 1$ be a real number such that $K^{1505} \lesssim |X|$.
Suppose
\[
\sum_{\xi\in X}E(A,\xi A) \geq \frac{|A|^3|X|}{K}.
\]
There exists an element $\bar a$ in $A$, a subfield $F\subseteq \F_q$, and an $F$-vector space $V\subseteq \F_q$ such that
\[
|V\cap (A-\bar a)|\gtrsim  \max \{ K^{-704}|V| , K^{-85} |A| \}.
\]

Further, there is an element $x_0$ in $X$ and subsets $A_1\subseteq A-\bar a$, $X'\subseteq x_0^{-1}X$ such that $F$ is the subfield generated by $X'$, $V$ is the $F$-vector space generated by $A_1$, and
\[
|A_1|\gtrsim K^{-85}|A|\andd |X'|\gtrsim K^{-5}|X|.
\]
Moreover, if $E(A,\xi A) \geq |A|^3/K$ for all $\xi \in X$ we have $|X'| \gtrsim K^{-4} |X|$  and only require $K^{1504} \lesssim |X|$.
\end{thmalph}

Theorem~\ref{thm:C} constitutes the biggest step towards the proof of Theorem~\ref{thm:A}. Its proof has two main parts. The first is, roughly speaking, to show that if
\[
\sum_{\xi\in X}E(A,\xi A) \geq \frac{|A|^3|X|}{K}
\]
then there exists an absolute constant $C$ and large subsets $A_1 \subseteq A$ and $X' \subseteq X$ such that $|A_1 + X' A_1| \leq K^C |A_1|$. This was proved by Bourgain~\cite[Proof of Theorem C]{Bourgain2009}, see also \cite[Proposition 2]{Bourgain-Glibichuk2011}. We prove a quantitatively stronger result in Proposition~\ref{prop:2} in Appendix~\ref{app:A}.

The second part in the proof of Theorem~\ref{thm:C} is to extract exact algebraic structure for $A$ and $X$ when $|A+XA|$ is nearly minimum. Roughly, if $A+XA$ is small, then $A$ has large relative density in the vector space it spans over the field generated by $X$. In Section~\ref{sec:thmC} we prove the following more precise result. 

\begin{thmalph}
  \label{pivot-thm:1}
Let $W$ be a subset of an $\F_q$-vector space, let $X \subseteq\F_q$ be a set of scalars, and let $F=\gen X$ be the subfield generated by $X$. If 
\begin{enumerate}
\item $|W+XW|\leq K_1 |W|$, 
\item $|W + W| \leq K_2 |W|$, 
\item $|W + x W| \leq K_3 |W|$ for all $x \in X$, and
\item $K_1^4 K_2 K_3^4 < |X|$,
\end{enumerate}
then
\[
|W|\geq \frac{1}{2K_1^2 K_3^2}|\vspan_F(W)|,
\]
where $\vspan_F(W)$ is the vector space over $F$ spanned by $W$.
\end{thmalph}

In Theorem~\ref{pivot-thm:1} we characterise the whole of $W$, whereas in Theorem~\ref{thm:C} only a large subset of $A$. This has to do with additive energy being a robust measure of additive structure (perturbing the set a little makes little difference), while cardinality is a delicate measure of additive structure (perturbing the set a little can make a big difference).

In the context of Theorem~\ref{thm:A}, where $|A|$ is $q^{2/3}$  (up to small powers of $q$)  and $|X|$ is  $q^{1/3}$ (up to small powers of $q$) , we can extract from Theorem~\ref{thm:C} even more information and show that the example mentioned above ($q = p^3$, $X$ is a subset of the subfield isomorphic to $\F_p$ and $A$ is a two-dimensional vector space over $\F_p$), is morally the only way that \eqref{eq:hypothesis bad} can occur.  
\begin{cor} 
   \label{cor:C for us}
Let $A,X\subseteq\F_q^*$ be sets and $K$ be a positive real number.
Suppose $K^{1504} \lesssim |X|$, $q^{1/4}K^4\lesssim |X|<q^{1/2}$, $q^{1/2}K^{85}\lesssim |A|\leq q^{2/3}$ and $E(A,\xi A) \geq |A|^3/K$ for all $\xi \in X$.

There exists a subfield $F$ of cardinality $q^{1/3}$, $\bar a$ in $A$, and a two-dimensional vector space $V$ over $F$ such that $|(A-\bar{a}) \cap V| \gtrsim K^{-85} |A|$.
\end{cor}

The proof of Theorem~\ref{thm:C} combines the work of Bourgain and Glibichuk with ideas of Glibichuk and Konyagin~\cite{Glibichuk-Konyagin2007}, Helfgott~\cite{helfgott2008growth,helfgott2011growth,helfgott2015growth}, Gill and Helfgott~\cite{gill2014growth}, and Tao~\cite{tao2015expansion}. In the background is the work of Ruzsa~\cite{Ruzsa1992} and Gowers~\cite{Gowers1998,Gowers2001}.

\subsubsection*{Section~\ref{sec:thmD}: Large subsets of a vector space}

By Theorem~\ref{thm:B} and Corollary~\ref{cor:C for us} we are left to deal with the case where $q$ is a cube, $F$ is the subfield of cardinality $q^{1/3}$, and a translate of $A$ has large intersection with a two-dimensional vector space over $F$.
To deal with such sets we combine basic algebra with a character sum argument of Hart, Iosevich, and Solymosi~\cite{HIS2007}.  

\begin{thmalph}
  \label{thm:D}
Let $q$ be a cube.
Suppose $F \subset \F_q$ is the subfield of cardinality $q^{1/3}$ and $V \subset \F_q$ is a two-dimensional vector space over $F$.
If $A$ is a subset of $V$ such that $|A| > \sqrt{2} |V|^{7/8} = \sqrt{2} q^{7/12}$, then
\[
|\pda| = |VV| > \frac{q}{2}.
\]
\end{thmalph}
One may compare Theorem~\ref{thm:D}, which we prove in Section~\ref{sec:thmD}, to the result of Hart, Iosevich and Solymosi~\cite[Theorem~1.4]{HIS2007} and a result of Bennett, Hart, Iosevich, Pakianathan and Rudnev~\cite[Corollary~1.8]{BHIPR2017}.

The former states that if $|A| > C q^{3/4}$ for some absolute constant $C$, then $\pda = \F_q$. Theorem~\ref{thm:D} is analogous because $\F_q$ is a one-dimensional vector space over itself and $\F_q \F_q = \F_q $. The exponent in Theorem~\ref{thm:D} is worse because we work in a two-dimensional vector space.

The latter states that if $|A| > C q^{2/3}$ for some absolute $C$, then $\pda$ determines a positive proportion of the elements of $\F_q$. Restricting $A$ to subsets of $V$ allows one to get a similar result for smaller sets.

To get the best possible lower bound on $|A|$ in Theorem~\ref{thm:D}  we apply Weil's bounds on Kloosterman sums. Kloosterman's bounds, which are easier to derive, could be used at no cost in the required lower bound on $|A|$ in Theorem~\ref{thm:A}.


\subsubsection*{Structure of the rest of the paper}

The derivation of Theorem~\ref{thm:A} from Theorems~\ref{thm:B}, Corollary~\ref{cor:C for us}, and Theorem~\ref{thm:D} is conducted in Section~\ref{sec:thmA}. 

In Section~\ref{sec:thmE} we prove a stronger result over prime order fields.
\begin{thmalph}
  \label{thm:E}
For all constants $L>0$ there exists a constant $c>0$ such that for all primes $p$ and all sets $A,B,C,D \subseteq \mathbb{F}_p$ that satisfy
\[
|A| |B| |C| |D| > L p^{12/5}
\]
and (assuming $|A| \leq |B|$, $|C| \leq |D|$, and $|B| \leq |D|)$
\[
|A| |B|^4 |C| \geq p^2 |D|^2,
\]
we have
\[
|(A-B)(C-D)| > c p.
\]
\end{thmalph}
This leads to expansion for rational functions like 
\[
f(w,x,y,z) = (w+x^2)(y-z^{-1})
\] 
with threshold $|A| \geq p^{3/5}$ (by simply noting $|A| |A^2| |A| |A^{-1}| \geq |A|^4/2$).

Bourgain's version of Theorem~\ref{thm:C}~\cite[Theorem~C]{Bourgain2009} over prime finite fields is a powerful result, with many applications in the literature.
In Section~\ref{sec:more on C} we discuss what is currently known and highlight the strong bounds that follow from Rudnev's point-plane incidence theorem~\cite{Rudnev}.

Appendix~\ref{app:A} is dedicated to the proof of a sharper version of a result of Bourgain and Glibichuk from~\cite{Bourgain-Glibichuk2011}. It roughly asserts that if~\eqref{eq:hypothesis bad} holds, then there are large subsets $A_1 \subset A$ and $X' \subset X$ such that $|A_1+X' A_1| \leq K^C |A'|$.

In Appendix~\ref{app:A}, we also prove a version of the \bsg{} theorem~\cite{Gowers1998,Gowers2001} that shows if $E(A,B)\geq (|A||B|)^{3/2}/K$, then $|A'+B'|\lesssim K^3(|A||B|)^{1/2}$ for subsets $A'\subseteq A$, $B' \subseteq B$ of relative density $\gtrsim 1/K$.
This is slightly worse than the bounds achieved by Schoen \cite{schoen2015bounds} under the hypothesis $E(A)\geq |A|^3/K$, however it applies to the case of distinct $A$ and $B$, and improves the bounds of Bourgain and Garaev \cite{bourgain2009variant}.

\subsection{Notation}

For positive quantities $X$ and $Y$, we use $X\ll Y$ to mean that $X\leq CY$ for some constant $C>0$ and we use $X\lesssim Y$ to mean that $X\ll (\log Y)^k Y$ for some constant $k>0$. The letter $p$ always denotes a prime and $q$ a prime power; $\F_{q}$ is the finite field with $q$ elements and $\F_{p}^d$ the $d$-dimensional vector space over $\F_{q}$. For $x,y \in \F_p$ we sometimes write $\tfrac{x}{y}$ instead of $xy^{-1}$ and implicitly assume $y \neq 0$. We use standard notation on set operations. For example we denote by $A+B = \{a+b : a\in A, b \in B\}$. We use representation functions like $r_{A+B}(x)$, which counts the number of ordered ways $x \in \F_q$ can be expressed as a sum $a + b$ with $a \in A$ and $b \in B$. Note that $A+B$ is the support of $r_{A+B}$. For a subset $G \subseteq A \times B$ we define $A\psum G B = \{ a+b : (a,b) \in G\}$. For a subspace $W$ of a finite field, we denote by $\vspan_F(W)$ the vector space over a subfield $F$ spanned by $W$ and by $\gen W$ the subfield generated by $W$.


\section{Arithmetic combinatorics}
\label{sec:Prelim}

In various stages of the proofs we utilise sumset cardinality inequalities that follow from the work of  of Pl\"unnecke and Ruzsa~\cite{Plunnecke1970,Ruzsa1978,Ruzsa1989,Ruzsa1992}. An excellent reference is~\cite{Ruzsa2009}.
\begin{lemma}[Sumset inequalities]
\label{pivot-lem:PR}
Let $A,B_1,\dots, B_h$ be finite non-empty sets in a commutative group. The following inequalities are true.
\begin{enumerate}
\item Pl\"unnecke's inequality for different summands~\cite[Theorem~1.6.1]{Ruzsa2009} (see also~\cite{Plunnecke1970,Ruzsa1989}).
\[
|B_1 + \dots + B_h| \leq \frac{|A+B_1|}{|A|} \cdots  \frac{|A+B_h|}{|A|} \; |A|.
\]
\item Pl\"unnecke's inequality for different summands with a large subset~\cite[Corollary~1.5]{katz2008slight} (see also~\cite[Theorem~1.7.1]{Ruzsa2009}). There exists a subset $Y \subseteq A$ of cardinality $|Y| \geq |A|/2$ such that
\[
|Y + B_1 + \dots + B_h| \leq 2^h \frac{|A+B_1|}{|A|} \cdots  \frac{|A+B_h|}{|A|} \; |A|.
\]
\item Ruzsa's triangle inequality~\cite[Theorem~1.8.1]{Ruzsa2009} (see also~\cite{Ruzsa1978})
\[
 |B_1 - B_2| \leq \frac{|A+B_1|}{|A|} \, \frac{|A+B_2|}{|A|} \; |A|.
\]
\item Pl\"unnecke-Ruzsa inequality for different summands (see also \cite[Theorem~1.1.1]{Ruzsa2009} and \cite{Plunnecke1970,Ruzsa1992,GPNAP}).
\[
|\underbrace{B_1\pm B_2 \pm \dots \pm B_h}_{h \text{ summands}}| \leq  \left( \prod_{i=1}^h \frac{|A+B_i|}{|A|}\right) \;  |A|.
\]
\end{enumerate}
\end{lemma}

We also make heavy use of the notion of so-called additive energy.

\begin{definition}
Let $A, B \subseteq F_q$ and $\xi \in\F_q^*$. The \emph{additive energy} of $A$ and $\xi B$ is defined in the following equivalent ways.
\begin{align*}
E(A, \xi B) 
& = \sum_{x \in\F_q} r_{A \pm \xi B}^2(x) \\
& = | \{ (a_1, a_2 , b_1, b_2) \in A \times A \times B \times B : (a_1-a_2) = \xi (b_1-b_2) \}.
\end{align*}
\end{definition}

The Cauchy-Schwarz inequality and the identity $\sum_x r_{A + \xi B}(x) = |A| |B|$ (each ordered pair $(a, b) \in A \times B$ contributes 1 to the sum) imply
\begin{equation}\label{eq:CS}
E(A,\xi B) \geq \frac{\left(\sum_x r_{A + \xi B}(x)\right)^2}{|\supp(r_{A + \xi B})|} \geq \frac{|A|^2 |B|^2}{q}.
\end{equation}

We will use repeatedly the following simple yet powerful lemma; see \cite[Lemma 2.1]{BKT2004},~\cite[Lemma~2]{Konyagin2003} and~\cite[Lemma 3]{BGK2006}. We need a straightforward generalisation of additive energy to $\F_q$-vector spaces. Given a set $A$ and a non-zero element $\xi$ in the vector space and a set $B \subseteq \F_q^*$,  we define the additive energy $E(A, \xi B)$ like in $\F_q$.

\begin{lemma}
  \label{lem:BKT}
Let $V$ be a vector space over $\F_q$ .
If $A$ and $S$ are finite subsets of $V$, with $0\not\in S$, and $B \subseteq \F_q^*$ is a collection scalars, then
\begin{equation*}
\sum_{\xi \in S} E(A, \xi B) \leq   |A|^2 |B|^2 + |S| |A| |B|.
\end{equation*}
Hence there exists $\xi \in S$ such that
\[
|A+ \xi B| \geq \frac{1}{2} \min \{ |S|, |A| |B| \}.
\]
\end{lemma}

The lemma applied to $V=\F_q$, $S = \F_q^*$, and $A=B$ implies that $E(A, \xi A) > |A|^3/K$ for at most $K q/|A|$ many $\xi$. The proof of this fact is implicit in the proof of Theorem~\ref{thm:B}.
The majority of the remainder of the paper is devoted to the task of refining this simple corollary of Lemma~\ref{lem:BKT}.


\section{Proof of Theorem B}
\label{sec:thmB}
In this section we prove Theorem~\ref{thm:B}, which we recall here.
\begin{repthm}{thm:B}
Let $A$ be a set in $\F_q$ and $D_\times(A)$ be the quantity
\begin{equation}
  \label{eq:24}
  D_\times(A) = |\{(a_1,\ldots,a_8) \in A^8 \colon (a_1-a_2)(a_3-a_4)=(a_5-a_6)(a_7-a_8) \}|.
\end{equation}
Suppose that there is a positive constant $K$ such that $|A| \leq  q/(4K)$ and
\begin{equation*}
  D_\times(A) \geq \frac{|A|^8}q + \frac{3q|A|^5}K.
\end{equation*}
There exists a set $X\subseteq\F_q$ such that $E(A,\xi A) \geq |A|^3/K$ for all $\xi \in X$ and
\[
\frac{q}{K|A|}\leq |X|\leq \frac{4K q}{3|A|}.
\]
\end{repthm}

\begin{proof}
Let
\[
r(x)=r_{(A-A)(A-A)}(x) = |\{(a_1,\ldots,a_4)\in A^4\colon (a_1-a_2)(a_3-a_4)=x\}|.
\]
Then
\[
D_\times(A) = \sum_{x \in \F_q} r^2(x).
\]
There are at most $4 |A|^6$ solutions where either side of \eqref{eq:24} is zero, thus
\begin{equation*}%
  \sum_{x \in \F_q} r^2(x) - \left|\left\{ \frac{a_1-a_2}{a_5-a_6}=\frac{a_7-a_8}{a_3-a_4}\not=0,\infty \right\}\right| \leq 4 |A|^6.
\end{equation*}%
For $\xi \in \F_q^*$, write
\[
 Q_{\xi} = \left|\left\{ (a_1,\ldots,a_4)\in A^4\colon\frac{a_1-a_2}{a_3-a_4}=\xi \right\}\right|,
 \]
so that
\begin{equation}
  \label{eq:26}
  \sum_{x \in \F_q} r^2(x) \leq \sum_{\xi \in \F_q^*}Q_\xi^2 + 4 |A|^6.
\end{equation}

Note that $Q_\xi = E(A,\xi A)-|A|^2$, so by Lemma~\ref{lem:BKT}
\begin{equation}
   \label{eq:sumQ}
\sum_{\xi \in \F_q^*}Q_\xi \leq |A|^4.
\end{equation}

If we set 
\[
E_\xi = E(A,\xi A) -\frac{|A|^4}{q},
\]
then by~\eqref{eq:CS} the quantity $E_\xi$ is \emph{positive} for all $\xi$.
For any subset $S\subseteq\F_q^*$, Lemma~\ref{lem:BKT} implies that
\begin{equation}
  \label{eq:28}
 \sum_{\xi \in S} E_\xi \leq  \sum_{\xi \in  \F_q^*}E_\xi \leq q|A|^2.
\end{equation}

Now we estimate the second moment of $Q_\xi$.
First we replace one power of $Q_\xi$ by $E_\xi$ and estimate the error using~\eqref{eq:sumQ}.
\begin{align*}
    \sum_{\xi \in \F_q^*}Q_\xi^2&= \sum_{\xi \in \F_q^*} Q_\xi \left( E(A,\xi A) - |A|^2 \right)\\
&= \sum_{\xi \in \F_q^*} Q_\xi \left( E_\xi +\frac{|A|^4}q - |A|^2 \right)\\
&= \sum_{\xi \in \F_q^*} Q_\xi E_\xi +\left(\frac{|A|^4}q - |A|^2 \right) \sum_{\xi \in \F_q^*} Q_\xi\\
&\leq \frac{|A|^8}q +\sum_{\xi \in \F_q^*} Q_\xi E_\xi.
\end{align*}
Hence by~\eqref{eq:26}
\begin{equation}
   \label{eq:31}
  D_\times(A) \leq \frac{|A|^8}q +\sum_{\xi \in \F_q^*} Q_\xi E_\xi + 4|A|^6.
\end{equation}

Now we estimate the remaining sum over $\xi$.
Let
\[
X = \{\xi \in \F_q^*\colon Q_\xi \geq |A|^3/K\}.
\]
Then, using \eqref{eq:28},
\begin{equation}
 \label{eq:I II}
\sum_{\xi \in \F_q^*} Q_\xi E_\xi \leq \sum_{\xi\in X}Q_\xi E_\xi+ \frac{|A|^3}K \sum_{\xi \in \F_q^*}E_\xi \leq  \sum_{\xi\in X} Q_\xi E_\xi+ \frac{q |A|^5}{K}.
\end{equation}

Combining the lower bound on $D_\times(A)$ in the hypothesis of the theorem, the upper bound $|A| \leq q/(4K)$ (which implies $4|A|^6 \leq q|A|^5/K$), and equations \eqref{eq:31} and \eqref{eq:I II}, we get
\begin{equation}
   \label{eq:12}
\sum_{\xi  \in X} Q_\xi E_\xi   \geq \frac{q |A|^5}{K}.
\end{equation}
By the trivial bound $Q_\xi E_\xi \leq |A|^6$ we derive $|X| \geq q/(|A| K)$. 

The final step is to obtain an upper bound on $|X|$.
The bound $|A| \leq q/(4K)$ implies that $E_\xi \geq Q_\xi - |A|^4/q \geq 3 |A|^3/(4K)$  for all $\xi \in X$.
By~\eqref{eq:28} we get
\[
\frac{3|A|^3}{4K} |X| \leq \sum_{\xi \in X} E_\xi \leq q|A|^2,
\]
which implies the desired bound $|X| \leq 4K q/(3|A|)$.
\end{proof}


\section{Proof of Theorems C and D}
\label{sec:thmC}

\subsection{Two auxiliary results}

The proof of Theorem~\ref{thm:C} relies on two auxiliary results.
The first roughly speaking asserts that if the additive energy $E(A,\xi A)$ is large for many $\xi$ in $X$, then $|A+X A|$ is small.

\begin{prop}
  \label{prop:2}
Let $A,X\subseteq\F_q^*$.
If there is a positive real number $K>0$ such that
\[
 \sum_{\xi\in X}E(A,\xi A) \geq \frac{|A|^3 |X|}K,
\]
then there exist elements $\bar a\in A$ and $x_0\in X$ and subsets
$A_1 \subseteq A - \bar{a} $ and $X'\subseteq x_0^{-1} X$ such
that:
\begin{enumerate}
\item[(i)] $\ds |A_1|\gtrsim K^{-85}|A|$ and $|X'|\gtrsim K^{-5}|X|$.
\item[(ii)] $ |A_1+X'A_1|\lesssim K^{226} |A_1|$.
\item[(iii)] $\ds |A_1 \pm A_1| \lesssim K^{92} |A_1|$.
\item[(iv)] $\ds |A_1 \pm x' A_1| \lesssim K^{126} |A_1|$ for all $x' \in X'$.
\end{enumerate}
Moreover, if $E(A,\xi A) \geq |A|^3 /K$ for all $\xi \in X$, then $|X'| \gtrsim K^{-4} |X|$.
\end{prop}

The proposition, which is proved in Appendix~\ref{app:A}, follows from a more careful analysis of the arguments of Bourgain in \cite[Theorem C]{Bourgain2009} and Bourgain--Glibichuk in \cite[Proposition 2]{Bourgain-Glibichuk2011}. The structure of our argument is very similar to Bourgain's, though the quantitative aspects of the bounds we obtain are stronger than what is found in \cite[Proposition 2]{Bourgain-Glibichuk2011}. 

The second auxiliary result is Theorem~\ref{pivot-thm:1}, which we restate here.  
\begin{repthm}{pivot-thm:1}
Let $W$ be a subset of an $\F_q$-vector space, let $X \subseteq\F_q$ be a set of scalars, and let $F=\gen X$ be the subfield generated by $X$. If 
\begin{enumerate}
\item $|W+XW|\leq K_1 |W|$, 
\item $|W + W| \leq K_2 |W|$, 
\item $|W + x W| \leq K_3 |W|$ for all $x \in X$, and
\item $K_1^4 K_2 K_3^4 < |X|$,
\end{enumerate}
then
\[
|W|\geq \frac{1}{2K_1^2 K_3^2}|\vspan_F(W)|.
\]
\end{repthm}

We also record the following specialization of Theorem~\ref{pivot-thm:1}, which will be used in the proof of Corollary~\ref{cor:C for us}.
\begin{cor}
\label{pivot-cor:C}
Let $q > 2^{12}$ be a prime power, let $W, X \subset \F_q$ be sets and $K_1, K_2, K_3 $ be a real numbers. Suppose that $q^{1/4} < |X| < q^{1/2}$ and $q^{1/2}<|W| \leq q^{2/3}$. If $|W+XW| \leq K_1 |W|$, $|W+W| \leq K_2 |W|$,  $|W + x W| \leq K_3 |W|$ for all $x \in X$, and $K_1^4 K_2 K_3^4 < |X|$, then there exists a subfield $F$ of cardinality $q^{1/3}$ and $v_1,v_2\in W \setminus \{0\}$ such that $W \subseteq F v_1  +F v_2$, $X \subseteq F$ and
\[
|W| \geq \frac{|F|^2}{2 K_1^2 K_2^2}.  
\]
\end{cor}

Corollary~\ref{pivot-cor:C} and Theorem~\ref{pivot-thm:1} are both proved further down in this section.
Let us now deduce Theorem~\ref{thm:C} and Corollary~\ref{cor:C for us}, assuming the auxiliary results.

\subsection[Proof of Theorem C and Corollary 1]{Proof of Theorem~\ref{thm:C} and Corollary~\ref{cor:C for us}}

\begin{repthm}{thm:C}
Let $A$ and $X$ be subsets of $\F_q^*$ and let $K\geq 1$ be a real number such that $K^{1505} \lesssim |X|$.
Suppose
\[
\sum_{\xi\in X}E(A,\xi A) \geq \frac{|A|^3|X|}{K}.
\]
There exists an element $\bar a$ in $A$, a subfield $F\subseteq \F_q$, and an $F$-vector space $V\subseteq \F_q$ such that
\[
|V\cap (A-\bar a)|\gtrsim  \max \{ K^{-704}|V| , K^{-85} |A| \}.
\]

Further, there is an element $x_0$ in $X$ and subsets $A_1\subseteq A-\bar a, X'\subseteq x_0^{-1}X$ such that $F$ is the subfield generated by $X'$ and $V$ is the $F$-vector space generated by $A_1$, and
\[
|A_1|\gtrsim K^{-85}|A|\andd |X'|\gtrsim K^{-5}|X|.
\]
Moreover, if $E(A,\xi A) \geq |A|^3/K$ for all $\xi \in X$ we have $|X'| \gtrsim K^{-4} |X|$ and only require $K^{1504} \lesssim |X|$.
\end{repthm}

\begin{proof}
We are given
\[
  \sum_{\xi\in X}E(A,\xi A) \geq \frac{|A|^3 |X|}K.
\]
By Proposition~\ref{prop:2} there exist elements $\bar a\in A$ and $x_0\in X$ and subsets $A_1\subseteq A-\bar a$ and $X'\subseteq x_0^{-1}X$ such that
\[
|A_1+X'A_1|\lesssim K^{226} |A_1|, |A_1+A_1| \lesssim K^{92} |A_1| , |A_1 + x' A_1| \lesssim K^{126} |A_1| \, \forall x' \in X',
\]
\[
|A_1|\gtrsim K^{-85}|A| \andd |X'|\gtrsim K^{-5}|X|.
\]

Since $1505 = (4 \cdot 226 + 92 + 4 \cdot 126) + 5$, the hypothesis $K^{1505} \lesssim |X|$ means 
\[
(K^{226})^4 K^{92} (K^{126})^4 \lesssim K^{-5} |X| \lesssim |X'|
\]
and we may apply Theorem~\ref{pivot-thm:1}. Letting $F$ be the field generated by $X'$, we get
\[
|A_1|\gtrsim K^{- 2 \cdot (226 + 126)} |\vspan_{F}(A_1)| = K^{- 704}|\vspan_{F}(A_1)|. 
\]
Setting $V=\vspan_{F}(A_1)$, the last equation implies that
\[
|V\cap (A-\bar a)|\gtrsim K^{-704}|V|.
\]
For the second lower bound on $|V\cap (A-\bar a)|$, observe that $|V\cap (A-\bar a)| \geq |A_1|$, since $A_1 \subseteq V, A-\bar a$.

The special case where $E(A,\xi A) \geq |A|^3 /K$ for all $\xi \in X$ is deduced similarly from the last part of Proposition~\ref{prop:2}.
\end{proof}

\begin{repcor}{cor:C for us}
Let $A,X\subseteq\F_q^*$ be sets and $K$ be a positive real number.
Suppose $K^{1504} \lesssim |X|$, $q^{1/4}K^4\lesssim |X|<q^{1/2}$, $q^{1/2}K^{85}\lesssim |A|\leq q^{2/3}$ and $E(A,\xi A) \geq |A|^3/K$ for all $\xi \in X$.

There exists a subfield $F$ of cardinality $q^{1/3}$, $\bar a$ in $A$, and a two-dimensional vector space $V$ over $F$ such that $|(A-\bar{a}) \cap V| \gtrsim K^{-85} |A|$.
\end{repcor}

\begin{proof}
Proceeding like in the proof of Theorem~\ref{thm:C}, only using the special case of Proposition~\ref{prop:2}, we get elements $\bar a\in A$ and $x_0\in X$ and subsets $A_1 \subseteq A-\bar a$ and $X'\subseteq x_0^{-1}X$ such that 
\[
|A_1+X'A_1|\lesssim K^{226} |A_1|, |A_1+A_1| \lesssim K^{92} |A_1| , |A_1 + x' A_1| \lesssim K^{126} |A_1| \, \forall x' \in X',
\]
\[
|A_1|\gtrsim K^{-85}|A| \andd |X'|\gtrsim K^{-4}|X|.
\]

As observed in the proof of Theorem~\ref{thm:C}, because $K^{1504} \lesssim |X|$, we may apply Corollary~\ref{pivot-cor:C} to $A'$ and $X'$ as long as
\begin{equation}
  \label{eq:33}
  q^{1/4} < |X'| < q^{1/2}\andd q^{1/2} < |A'|\leq q^{2/3}.
\end{equation}
The condition $q>2^{12}$ is irrelevant in this context. Assuming \eqref{eq:33} for now, we conclude that there is a subfield $F$ of cardinality $q^{1/3}$ and a two-dimensional vector space $V$ over $F$ such that $A_1 \subseteq V$. Therefore
\[
|(A-\bar{a}) \cap V| \geq  |A_1| \gtrsim K^{-85} |A|,
\]
as desired.

To ensure \eqref{eq:33}, it suffices to take
\[
q^{1/4}K^4\lesssim |X| < q^{1/2}\andd  q^{1/2}K^{85}\lesssim |A|\leq q^{2/3}.
\]
\end{proof}

\subsection{A pivot argument for vector spaces}
\label{sec:pivot}

We now prove Theorem~\ref{pivot-thm:1}, which is a specialization of Gill and Helfgott's pivot theorem \cite{gill2014growth}.
Our approach is inspired by Tao's concept of involved elements \cite[Section 1.5]{tao2015expansion}.
Throughout the remaining of this section $\F$ is a finite field and $V$ is a vector space over $\F$.

For a subset $X\subseteq \F$ we denote by $\gen X$ the subfield generated by $X$. We use ${\gen X}_+$ and ${\gen X}_\times$ to denote the additive group and the multiplicative group (union $\{0\}$) generated by $X$, respectively, so that 
\[
\gen X = \gen{\gen X_\times}_+.
\]
(This is because $\gen{\gen X_\times}_+$ contains $X\cup\{0,1\}$, is closed under addition, and is closed under multiplication, since products of sums of products are sums of products; since $\gen{\gen X_\times}_+$ is finite, closure implies the existence of inverses.)
For a subset $W\subseteq V$ we denote by $\vspan_{X}(W)$ the vector space over $\gen X$ generated by the vectors in $W$. The key definition is the following.

\begin{definition}[Pivots and involved elements]
Let $V$ be a vector space over a finite field $\F$, $W \subseteq V$ and $X \subseteq \F$. For $\xi$ in $V$, define $\phi_\xi\colon W\times X\to V$ by $\phi_\xi(v,\alpha)=v+\alpha\xi$. We say that $\xi$ is a \emph{pivot}\/ for $(W,X)$ if $\phi_\xi$ is injective.

An element is \emph{involved}\/ with the sets $W$ and $X$, if it is not a pivot. We denote by $\mathcal{I}_{W,X}$ the set of involved elements.
\end{definition}

The rest of this section is devoted to showing that the collection of involved elements has an exact algebraic structure (namely, it is a vector space) assuming that $W$ and $X$ have weak algebraic structure (namely, $|W+XW|$ is close to minimum).

The next lemma shows that involved elements have low complexity in terms of $W$ and $X$.
\begin{lemma}[Properties of involved elements] 
\label{pivot-lem:pivotProp}
Let $V$ be a vector space over a finite field $\F$, $W \subseteq V$ and $X \subseteq \F$. A vector $ \xi \in V$ is an involved element for $(W,X)$ if and only if
\begin{itemize}
\item there exist distinct scalars $\alpha_1 \neq \alpha_2\in X$ and vectors  $v_1,v_2\in W$ such that $\xi = (\alpha_1-\alpha_2)^{-1}(v_2-v_1)$;
\item $|W+X\xi| < |W||X|$.
\end{itemize}
\end{lemma}

\begin{proof}
Both parts follow from the non-injectivity of $\phi_\xi$.

For the first part note that there exist distinct $(v_1, \alpha_1), (v_2, \alpha_2) \in W \times X$ such that $\phi_\xi(v_1, \alpha_1) = \phi(v_2, \alpha_2)$ precisely when $v_1 - v_2 = \xi (\alpha_1-\alpha_2)$. We cannot have $\alpha_1 = \alpha_2$ (else $v_1=v_2$).

For the second part note $ |W+ X \xi| = |\phi_\xi(W,X)|$.
\end{proof}

Lemma~\ref{pivot-lem:pivotProp} implies that $\mathcal{I}_{W,X}\subseteq \vspan_{X}(W)$, since every involved element $\xi$ can be written as $\xi = (\A_1-\A_2)^{-1}(v_2-v_1)$ with $\A_i\in X,v_i\in W$.
The following proposition proves the reverse inclusion when $|W+XW|$ and $|W+W|$ are sufficiently small.

\begin{proposition}
\label{pivot-prop:NPspan}
Let $V$ be a vector space over a finite field $\F$, $W \subseteq V$ and $X \subseteq \F$. In the notation used in this section, if
\begin{enumerate}
\item $|W+XW|\leq K_1 |W|$, 
\item $|W + W| \leq K_2 |W|$, 
\item $|W + x W| \leq K_3 |W|$ for all $x \in X$, and
\item $K_1^4 K_2 K_3^4 < |X|$,
\end{enumerate}
then $\vspan_{X}(W)\subseteq \mathcal{I}_{W,X}$ and therefore $\vspan_{X}(W) = \mathcal{I}_{W,X}$.
\end{proposition}

In the proof of Proposition~\ref{pivot-prop:NPspan} we will repeatedly use a key observation: if $\xi$ is an involved element, then we may strengthen the bound $|W+X\xi|<|X||W|$ by exploiting the low complexity of $\xi$.
\begin{lemma}
\label{pivot-lem:pivotBnd}
Let $V$ be a vector space over a finite field $\F$, $W \subseteq V$ and $X \subseteq \F$. Suppose that $|W+XW|\leq K_1 |W|$ and $|W + x W| \leq K_3 |W|$ for all $x \in X$. If $\xi$ is an involved element, then
\[
|W+X\xi|\leq K_1^2 K_3^2 |W|.
\]
\end{lemma}
\begin{proof}
Since $\xi$ is an involved element, there exist $x_1,x_2\in X, w_1,w_2\in W$ such that
  \begin{equation*}
\xi = (x_1-x_2)^{-1}(w_2-w_1).
  \end{equation*}
Thus
\[
(x_1-x_2)[ W+X\xi] = (x_1-x_2)W + X(w_2-w_1)\subseteq x_1 W- x_2 W+ XW- XW. 
\]
By the Pl\"unnecke--Ruzsa inequality for different summands in Lemma~\ref{pivot-lem:PR} we have
\[
|W + X \xi | \leq \left(\frac{|W+x_1 W|}{|W|}\right)  \left(\frac{|W + x_2 W|}{|W|}\right)  \left(\frac{|W+WX|}{|W|}\right)^2 |W| \leq K_1^{2} K_3^2 |W|.
\]
\end{proof}

We apply Lemma~\ref{pivot-lem:pivotBnd} to prove that the set of involved elements is closed under both addition and scalar multiplication, under suitable conditions on sumset cardinalities.

\begin{lemma}[Closure of involved elements under addition]
\label{pivot-lem:NPadd}
Let $V$ be a vector space over a finite field $\F$, $W \subseteq V$ and $X \subseteq \F$. If $|W+XW| \leq K_1|W|$, $|W+W| \leq K_2 |W|$, and $|W + x W| \leq K_3 |W|$ for all $x \in X$, where $K_1^4 K_2 K_3^4 <|X|$, then $\mathcal{I}_{W,X}\pm \mathcal{I}_{W,X} \subseteq \mathcal{I}_{W,X}$.
\end{lemma}
\begin{proof}
Let $\xi_1$ and $\xi_2$ be elements of $\mathcal{I}_{W,X}$. By Lemma~\ref{pivot-lem:pivotProp}, to show that $\xi_1 + \xi_2 \in \mathcal{I}_{W,X}$ we must establish $|W + X(\xi_1 + \xi_2)| < |X| |W|$. 

Note $|W+X(\xi_1+\xi_2)|\leq |W+X\xi_1+X\xi_2|$. To bound this sumset we use Pl\"unnecke's inequality for different summands in Lemma~\ref{pivot-lem:PR} and Lemma~\ref{pivot-lem:pivotBnd}.
\[
|W + X \xi_1 + X \xi_2| \leq  \frac{|W+W|}{|W|} \, \frac{|W+X \xi_1|}{|W|} \, \frac{|W+X \xi_2|}{|W|} \; |W| \leq K_1^4 K_2 K_3^4 |W|.
\]
Thus if $K_1^4 K_2 K_3^4 <|X|$, we have the desired inequality
\[
|W+X(\xi_1+\xi_2)| < |X||W|.
\]

We must also show that $\xi\in \mathcal{I}_{W,X}$ if and only if $-\xi\in \mathcal{I}_{W,X}$. It is enough to prove that if $\xi \in \mathcal{I}_{W,X}$, then $-\xi \in \mathcal{I}_{W,X}$. This follows from Lemma~\ref{pivot-lem:PR}, Lemma~\ref{pivot-lem:pivotBnd} and Lemma~\ref{pivot-lem:pivotProp}
\[
|W+X(-\xi)| \leq \frac{|W+W| |W+ X \xi|}{|W|} \leq K_1^2 K_2 K_3^2   |W| < |X| |W|.
\] 
\end{proof}

\begin{lemma}[Closure of involved elements under scalar multiplication]
\label{pivot-lem:NPmult}
Let $V$ be a vector space over a finite field $\F$, $W \subseteq V$ and $X \subseteq \F$.   If $|W+XW| \leq K_1|W|$, $|W+W| \leq K_2 |W|$, and $|W + x W| \leq K_3 |W|$ for all $x \in X$, where $K_1^2 K_3^3 < |X|$, then $\mathcal{I}_{W,X}$ is closed under multiplication by $\gen X_\times$.
\end{lemma}
\begin{proof}
By Lemma~\ref{pivot-lem:pivotProp} it is enough to show that, if $\xi$ is an involved element and $\alpha\in X$, then
\begin{equation*}%
   |W+X(\alpha\xi)|\leq K_1^{2} K_3^3 |W| \andd |W+X(\alpha^{-1}\xi)|\leq K_1^2 K_3^3 |W|.
\end{equation*}%
Indeed,
\[
|W+X(\alpha\xi)| \leq \frac{|W+\alpha W||\alpha W+ X(\alpha\xi)|}{|\alpha W|} = \frac{|W+ \alpha W | | W+ X \xi|}{|W|}  \leq K_1^2 K_3^3 |W|
\]
and, similarly,
\[
|W+X(\alpha^{-1}\xi)|=|\alpha W+X\xi|\leq\frac{|W+\alpha W||W+X\xi|}{|W|}\leq K_1^2 K_3^3 |W|.
\]
It follows by induction that if $\alpha_1,\ldots,\alpha_n\in X$ and $\xi\in \mathcal{I}_{W,X}$, then \[\alpha_1^{\pm 1}\cdots\alpha_n^{\pm 1}\xi\in \mathcal{I}_{W,X}.\]
\end{proof}

We now prove Proposition~\ref{pivot-prop:NPspan}.

\begin{proof}[Proof of Proposition~\ref{pivot-prop:NPspan}]
We want to show that the set $\mathcal{I}_{W,X}$ of involved elements contains $\vspan_{X}(W)$, where $\gen X$ is the subfield of $\F$ generated by $X$.

First, we show that $\mathcal{I}_{W,X}$ contains $W$.
If $\xi\in W$, then
\[
|W+X\xi|\leq |W + XW| < K_1 |W| < |X||W|.
\]
Hence $\xi$ is an involved element by Lemma~\ref{pivot-lem:pivotProp}.

Now by Lemma~\ref{pivot-lem:NPmult}, we have $\gen X_\times W\subseteq \mathcal{I}_{W,X}$, and hence by Lemma~\ref{pivot-lem:NPadd} we have for sums of any finite length
\[
\pm\gen X_\times W\pm \cdots \pm\gen X_\times W\subseteq \mathcal{I}_{W,X},
\]
hence
\[
\gen{\gen X_\times}_+W\subseteq \mathcal{I}_{W,X}.
\]
Further,
\[
\gen X W +\cdots+ \gen X W\subseteq \mathcal{I}_{W,X}
\]
for sums of any finite length, which implies that $\vspan_{X} (W)\subseteq \mathcal{I}_{W,X}$.
\end{proof}

\subsection[Proof of Theorem C and Corollary 1]{Proof of Theorem \ref{pivot-thm:1} and Corollary~\ref{pivot-cor:C}}

\begin{proof}[Proof of Theorem~\ref{pivot-thm:1}]
Proposition~\ref{pivot-prop:NPspan} implies that $\mathcal{I}_{W,X}=\vspan_F(W)$.
Hence by Lemma~\ref{lem:BKT} applied to $S = \vspan_F(W)$, there is an involved element $\xi_0$ such that
\[
|W+X\xi_0|\geq\frac 12\min(|\vspan_F(W)|, |X||W|).
\]
On the other hand, since $\xi_0$ is an involved element, Lemma~\ref{pivot-lem:pivotBnd} implies that
\[
|W+X\xi_0|\leq K_1^2 K_3^2|W| \leq |X|^{1/2} |W| < |X| |W| /2,
\]
because $|X| > K_1^4 K_2 K_3^4 \geq 1$. Therefore we have
\[
\frac 12 |\vspan_F(W)|\leq |W+X\xi_0|\leq K_1^2 K_3^2 |W|,
\]
which proves the theorem.
\end{proof}

Since $W\subseteq\vspan_F(W)$, the lower bound on $|W|$ is fairly tight. The inclusion $X\subseteq F$ implies $|X|\leq |F|$. The next step is to show that, under suitable cardinality constraints on $X$ and $W$, $W$ must have additional structure.


\begin{proof}[Proof of Corollary~\ref{pivot-cor:C}]
Let $q=p^r$ and $F = \gen{X}$, so that $|F|=p^{k}$ for $k\mid r$. Also let $d=\dim(\vspan_F(W))$. Our aim is to show that $d=2$ and $|F|=q^{1/3}$.

Theorem~\ref{pivot-thm:1} gives $|W|\geq K_1^{-2} K_3^{-2} |\vspan_F(W)| /2 $. This implies
\begin{equation}
  \label{pivot-eq:7}
  |W|^{1/d} \leq |F| \leq (2K_1^2 K_3^3 |W|)^{1/d}.
\end{equation}

We know that $|F| \geq |X| > q^{1/4}$ and so $|F|$ must be one of $q^{1/3}, q^{1/2}$ or $q$. 

Note next that $d >1$, because $d=1$ means $F = \vspan_F(W)$. Therefore
\[
\frac{|F|}{2 K_1^2 K_3^2} \leq |W| \leq |F|.
\]
The constrains on $|W|$ and $|F|$ immediately prohibit $F$ from being anything but $\F_q$. However, $F = \F_q$ is not possible either because it would imply
\[
\frac{q^{3/4}}{2}  \leq \frac{q}{2|X|^{1/2}} < \frac{q}{2 K_1^2 K_3^3} \leq |W| \leq q^{2/3}.
\]

Thus $d \geq 2$ and inequality~\eqref{pivot-eq:7}  implies 
\[
|F| < (2 |X|^{1/2} |W|)^{1/2} < (2 q^{11/12})^{1/2} \leq q^{1/2}.
\]
Therefore $|F| = q^{1/3}$.

Going back to~\eqref{pivot-eq:7} we get
\[
q^{d/3} \leq 2 K_1^2 K_3^2 |W| < 2 q^{11/12} \leq q .
\]
Therefore $d<3$ and so $d=2$. 

We conclude $W \subseteq \vspan_F(W) = F v_1 + F v_2$
for linearly independent $v_1,v_2 \in W$.
\end{proof}

\section[Proof of Theorem E]{Proof of Theorem~\ref{thm:D}}
\label{sec:thmD}

The set up for this section is as follows: $q$ is a cube, $F$ is the subfield of cardinality $q^{1/3}$, $V$ is a two-dimensional vector space over $F$, and $A \subseteq V$. Theorem~\ref{thm:D} asserts that if $|A| \gg q^{7/12}$, then $|\pda| > q/2 $. It is proved in two steps: first we prove $|VV| > q/2$; then we prove that if $|A| > \sqrt{2} q^{7/12}$, then $\pda = VV$.

\subsection{Basic facts}

We begin with some basic facts about $\F_q$, $V$ and $\hat{V}$, the dual group of $V$ viewed as a commutative group under addition.

Let $\xi \in V \setminus F$. Note that $\{1,\xi\}$ is a linearly independent set in $V$ and forms a basis of $V$. Hence $V = F + F \xi $. Moreover, $\{1,\xi,\xi^2\}$ is linearly independent in $\F_q$ (if $\xi^2$ were a non-trivial linear combination of $\{1,\xi\}$, then $V$ would be closed under multiplication and therefore would be a subfield of cardinality $q^{2/3}$). This means that $\F_q = F +F \xi + F \xi^2$. 

Suppose now that $V$ is a vector space over a field $F$. Let $\psi \in \hat{F}$ be a non-trivial additive character of $F$. The dual group $\hat{V}$ of additive characters of $V$ consists of characters of the form $v \mapsto \psi(v\cdot m)$, where $m$ is an element of $V$.
The orthogonality of characters implies that
\[
\frac{1}{|V|} \sum_{m \in V} \psi( v \cdot m) = \left\{ \begin{array}{lll} 1 & \mbox{if} & v = 0 \\ 0 & \mbox{if} & v \neq 0 \end{array} \right. . 
\]
For a set $A\subseteq V$ we denote by $A(v)$ the indicator function on $A$ and write
\[
\hat A(m) = \sum_{v\in V}A(v)\psi(-v\cdot m) = \sum_{v\in A}\psi(-v\cdot m)
\]
for the Fourier coefficients of $A$.

\subsection{The product set of the vector space}

\begin{proposition}\label{VV}
Let $q$ be a cube. Suppose $F \subset \F_q$ is the subfield of cardinality $q^{1/3}$ and $V \subset \F_q$ is a two-dimensional vector space over $F$.
\begin{enumerate}
\item[(i)] In odd characteristic: $\ds |VV| = \frac{|F|^3 + 2 |F|^2-|F|}{2}$.
\item[(ii)] In even characteristic: $\ds |VV| = \frac{|F|^3 + |F|^2}{2}$.
\end{enumerate}
In particular, $|V| >q/2$ and $|VV| = (1+o(1)) q/2$.
\end{proposition}

\begin{proof}
We first note $0 \in VV$. For non-zero elements $a + b \xi + c \xi^2 \in \F_q$ we seek $u, v \in V$ such that $uv = a + b \xi + c \xi^2$. After scaling we may assume that $u = 1 + c z^{-1} \xi$ and $v = a + z \xi$ for some $z \in F^*$. Next we observe that
\[
 a + b \xi + c \xi^2 = (1 + c z^{-1} \xi)(a + z \xi)
\] 
precisely when $ z^2 - b z + ac = 0$. This means that $a + b \xi + c \xi^2 \in VV$ precisely when the quadratic monic polynomial $z^2 - bz + ac$ has a root in $F$ (the case $a=b=c=0$ is included here). We treat the cases of even and odd characteristic separately. 

\underline{Odd characteristic}: The quadratic polynomial above has a root precisely when its discriminant $b^2-4ac$ is a square. We set $S = \{ x^2 : x \in F\}$ to be the set of squares and count the number of ordered triples $(a,b,c) \in F \times F \times F$ such that $b^2 - 4ac \in S$. To this end, for $y \in F$, we write $r(y) = | \{(a,c) \in F \times F : y = 4ac\}|$ for the number of ordered representations of $y$ as a product $4ac$ with $a, c \in F$. Observe that $r(y) = |F|-1$ for non-zero $y$ and $r(0) = 2|F|-1$.

The number of ordered triples $(a,b,c) \in F \times F \times F$ such that $(a,b,c) \in VV$ equals
\begin{align*}
\sum_{b \in F} \sum_{s \in S} r(b^2-s) 
& = \sum_{b \in F} \left( r(0) +  \sum_{s \in S \setminus \{b^2\} } r(b^2-s) \right) \\
& = |F| ( 2|F|-1 + (|F|-1) (|S|-1)) \\
 & = |F| (|F| |S| + |F| - |S|)  \\
& = \frac{|F|^3 + 2 |F|^2-|F|}{2}.
\end{align*}

\underline{Even characteristic}:
Let $|F| = 2^n$.
Recall that the Galois group of $F$ over $\F_2$ is generated by $\sigma\colon x\mapsto x^2$.
The trace function $\Tr: F \to \F_2$ is defined by
\[
\Tr(x)= \sum_{i=0}^{n-1} \sigma^i(x) = x + x^2 + x^4 + \dots + x^{2^{n-1}}.
\]
Since $\sigma$ is $\F_2$-linear, we may view $\Tr$ as a $\F_2$-linear functional on $F$.
Let $S = \{ x \in F : \Tr(x) = 0\}$ be its kernel. 
In particular note that $|S| = |F|/2$ because $\Tr$ is non-trivial.

To solve the equation
\begin{equation}
  \label{eq:4}
  z^2 + bz+ac =0
\end{equation}
we consider two cases.

First, suppose that $b=0$.
Then \eqref{eq:4} becomes $z^2=ac$, which has the unique solution $z=\sigma^{-1}(ac)$.

Second, suppose that $b\not=0$.
Then by replacing $z$ with $z/b$ we may put \eqref{eq:4} in the form
\begin{equation}
  \label{eq:11}
  z^2 + z + \frac{ac}{b^2}=0 \implies z^2+z = \frac{ac}{b^2}.
\end{equation}
Note that $z^2+z=z-\sigma(z)$.
Hilbert's Theorem 90 \cite[Theorem 6.3]{lang2002algebra} states that for cyclic extensions, $w=z-\sigma(z)$ if and only the trace of $w$ is zero.
Hence \eqref{eq:11} has a solution if and only if $ac/b^2\in S$.

To summarize: the quadratic polynomial has a root precisely when $b=0$ or $ac \in b^2 S$.

Writing $r(y) = | \{(a,c) \in F \times F : y = ac\}|$ for the number of ordered representations of $y$ as a product $ac$ with $a, c \in F$ and noting that $r(y) = |F|-1$ for non-zero $y$ and $r(0) = 2 |F|-1$ we see that the number of ordered triples $(a,b,c) \in F \times F \times F$ such that $(a,b,c) \in VV$ equals
\begin{align*}
|F|^2 + \sum_{b \in F^*} \sum_{s \in S} r(b^2 s) 
& = |F|^2 + \sum_{b \in F^*} \left( r(0) +  \sum_{s \in S \setminus \{0\} } r(b^2s) \right) \\
& = |F|^2 + (|F|-1)  ( 2|F|-1 + (|F|-1) (|S|-1)) \\
& = |F| (|F| |S| + |F| - |S|)  \\
& = \frac{|F|^3 + |F|^2}{2}.
\end{align*}
\end{proof}

\subsection{Large subsets of vector spaces}

The second step in the proof of Theorem~\ref{thm:D} is to prove that if $|A| \gg V^{7/8} = q^{7/12}$, then $\pda = VV$. We begin with a generalisation of a result of Hart, Iosevich, and Solymosi~\cite[Theorem 1.4]{HIS2007}. Its importance in our setting will become clear in the forthcoming corollary.

\begin{proposition}\label{prop:gen HIS}
Let $V$ be a vector space over a field $F$, $u,v \in V \sm \{0\}$ be non-zero vectors and $A \subseteq V$ be a set in $V$. Suppose $|A| \geq \sqrt{2} |V| / |F|^{1/4}.$ There exists $x \in F^*$ such that both $xu$ and $x^{-1}v$ belong to $A-A$.
\end{proposition}

\begin{proof}
Writing
\[
r_{A-A}(t)  = |\{(a,b) \in A \times A : t = a-b \}|
\]
for the number of ordered ways to express $t \in V$ as difference in $A -A$, we must prove
\begin{equation}\label{criterion}
\sum_{x \in F^*} r_{A-A}(x u) r_{A-A}(x^{-1} v) > 0.
\end{equation}

Recall that $\psi \in \hat{F}$ is a non-trivial additive character of $F$ and that, for all $m \in V$, the function $v \to \psi( m \cdot v)$ is an additive character of $V$. By character orthogonality
\begin{align*}
  r_{A-A}(t) &= \frac{1}{|V|} \sum_{m \in V} \sum_{a,b \in A} \psi((t + a - b) \cdot m) \\
&= \frac{1}{|V|} \sum_{m \in V}  \left| \sum_{a \in A} \psi(a \cdot m) \right|^2 \psi(t \cdot m)\\
&= \frac 1{|V|}\sum_{m\in V} |\hat A(m)|^2\psi(t\cdot m).
\end{align*}
Therefore, after some rearranging, the left side of~\eqref{criterion} equals
\begin{align} \label{eq:2}
\frac{1}{|V|^2} \sum_{m \in V}\sum_{m' \in V} |\hat A(m)|^2|\hat A(m')|^2\left(\sum_{x \in F^*} \psi(x (u \cdot m) + x^{-1} (v \cdot m'))\right).
\end{align}
If $ m \cdot u \neq 0$ or $m ' \cdot v\neq 0$, then the bracketed sum is a Kloosterman sum or a Ramanujan sum and its modulus is at most $2 |F|^{1/2}$, see~\cite[Theorem 11.11]{Iwaniec-Kowalski2004} and \cite{Conrad2002} for even characteristic.
Otherwise, when $m \cdot u = m' \cdot v =0 $, the bracketed sum equals $|F|-1$.
Substituting in~\eqref{eq:2}, we see that the left side of~\eqref{criterion} equals
\begin{align*}
& \frac{|F|-1}{|V|^2} \sum_{\substack{m \cdot u = 0 \\ m' \cdot v = 0}} |\hat A(m)|^2|\hat A(m')|^2+ \\
& \quad \frac{1}{|V|^2} \sum_{\substack{m \cdot u \neq 0 \\ \text{ or } \\ m' \cdot v \neq 0}} |\hat A(m)|^2|\hat A(m')|^2\left(\sum_{x\in F^*} \psi(x (u \cdot m) + x^{-1} (v \cdot m'))\right).
\end{align*}
The first summand is bounded from below by the contribution coming from $m = m'=0$ (noting that the remaining terms are all positive), which is $(|F|-1)|A|^4 / |V|^2$. The modulus of the second summand is at most
\[
2 |F|^{1/2} \frac{1}{|V|^2} \sum_{m \neq 0}\sum_{m' \neq 0}|\hat A(m)|^2|\hat A(m')|^2
= 2 |F|^{1/2} \left( |A| - \frac{|A|^2}{|V|} \right)^2.
\]
Therefore the left side of~\eqref{criterion} is positive 
if
\[
\frac{(|F|-1) |A|^{4}}{|V|^2} > 2 |F|^{1/2}  \left( |A| - \frac{|A|^2}{|V|} \right)^2,
\]
which holds when $|A|^2 |F|^{1/2} > 2 |V|^2$.
\end{proof}


\begin{corollary}\label{Large A in V}
Let $\F_q$ be a finite field, $F$ a subfield of $\F_q$, $V \subseteq \F_q$ be a vector space over $F$, and $A \subseteq V$ be a set in $V$. If $|A| \geq \sqrt{2} |V| / |F|^{1/4}$, then $\pda = VV$. 

In particular, if $\F_q$ is a cubic extension of $F$ and $V$ is a two dimensional vector space over $F$, then $|A| \geq \sqrt{2} |V|^{7/8} = \sqrt{2} q^{7/12}$ implies $\pda = VV$.
\end{corollary}

\begin{proof}
The origin $0$ belongs to both $(A-A)(A-A)$ and $VV$.
Given $w \in VV \sm \{0\}$, there exist $u, v \in V \sm \{0\}$ such that $w = uv$. By Proposition~\ref{prop:gen HIS} there exists $x \in F^*$ such that  both $xu, x^{-1} v \in A-A$. Therefore $ w = uv = (xu) (x^{-1} v) \in \pda$.
\end{proof}

\begin{remark}
The proof has the curious characteristic that we use additive characters of $V$ to deduce something about elements of $\F_q$ that may not belong to $V$.
\end{remark}




\section{Proof of Theorem A}
\label{sec:thmA}

At last we are ready to prove Theorem~\ref{thm:A}. 

\begin{repthm}{thm:A}
For all $q$ and all sets $A \subseteq \F_q$ that satisfy $\ds |A| \gtrsim q^{\tfrac 2 3 - \tfrac 1 {13,542}}$, we have $|\pda| > q/2$.
\end{repthm}

We consider separately two cases. The first is loosely speaking the case where $D_\times(A)$ is large and is where Theorem~\ref{thm:B}, Corollary~\ref{cor:C for us}, and Theorem~\ref{thm:D} are used. The second case, where $D_\times(A)$ is small, only requires a simple second moment argument. 

A peculiarity of the argument is that the first case leads to a better lower bound for $|A|$. The explanation for this is that in order to apply Corollary~\ref{cor:C for us} we need $D_\times(A) - |A|^8/q$ to be very large. Once we assume this and apply the corollary, we are in the regime of Theorem~\ref{thm:D} and the rest of the argument is fairly efficient. The price for assuming that $D_\times(A) - |A|^8/q$ is large is paid when applying the complimentary second moment argument. It is not possible to balance the two parts of the argument because the first part stops working before it can it be balanced with the second.

\subsection{Case I: Large $D_\times(A)$}
\label{sec:large D}

We assume that 
\[
D_\times(A) \geq \frac{|A|^8}{q} + \frac{3 q |A|^5}{K} \text{ for } K \lesssim \left(\frac{q}{|A|}\right)^{1/1505}.
\]

To be able easily confirm that the various mild constraints appearing in Theorem~\ref{thm:B}, Corollary~\ref{cor:C for us}, and Theorem~\ref{thm:D} are indeed satisfied, we assume that  $q^{3/5} \leq |A| \leq q^{2/3}$ (the result of Bennett, Hart, Iosevich, Pakianathan, and Rudnev~\cite{BHIPR2017} deals with larger $|A|$).
This gives the generous bound $K \lesssim q^{1/3,000}$, which is adequate for this purpose.

In particular, we have $|A| \ll q/K$. Theorem~\ref{thm:B} gives a set $X$ such that $E(A,\xi A) \geq |A|^3 /K$ for all $\xi \in X$ and
\[
\frac{q}{K |A|} \leq |X| \leq \frac{4 K q}{3|A|} \ll K q^{2/5} < q^{1/2}.
\]

We now apply Corollary~\ref{cor:C for us}. The constraints on $|X|$ and $|A|$ are satisfied because of the constraints on $|A|$ and $|X|$ and $K$. The, more restrictive, condition $K^{1504} \lesssim |X|$ is satisfied because of the assumption $K^{1505} \lesssim q/|A|$:
\[
K^{1504} \lesssim \frac{q}{K |A|} \leq |X|.
\]
Corollary~\ref{cor:C for us} now gives a subfield $F$ of cardinality $q^{1/3}$, a two-dimensional vector space $V$ over $F$ and $\bar{a} \in A$ such that $|(A-\bar a) \cap V| \gtrsim K^{-85} |A|$.

The final step is to apply Theorem~\ref{thm:D} to the set $A_{\bar a} = (A-\bar a) \cap V$. If $|A_{\bar a}| \gg |V|^{7/8} = q^{7/12}$, then
\[
|\pda| \geq |(A_{\bar a} - A_{\bar a}) (A_{\bar a} - A_{\bar a})| > \frac{q}{2}.
\]
The condition $|A_{\bar a}| \gg  q^{7/12}$ is satisfied when
\[
|A| \left( \frac{|A|}{q} \right)^{85/1505} \gtrsim q^{7/12} \impliedby |A|  \gtrsim q^{2,311/3,816}.
\]
Note here that $2,311/3,816 < 2/3 - 1/13,542$ and so the theorem holds in this case. 

\subsection{Case II: Small $D_\times(A)$}
\label{sec:small D}

We assume that 
\[
D_\times(A) \leq \frac{|A|^8}{q} + \frac{3 q |A|^5}{K} \text{ for } K \gtrsim \left(\frac{q}{|A|}\right)^{1/1505}.
\]
In particular, if $ |A| \gtrsim q^{3,009/4,514} = q^{2/3 - 1/13,542}$, then 
\[
 \frac{3 q |A|^5}{K} \leq \frac{|A|^8}{q}.
\]

A few basic observations are now adequate: $\pda$ is the support of the function $r_{\pda}$ and 
\[
\sum_{x \in \F_q}  r_{\pda}(x) = |A|^4 \andd \sum_{x \in \F_q}  r_{\pda}(x)^2 = D_\times(A).
\]
The Cauchy-Schwarz inequality implies
\[
|\pda| \geq \frac{\left( \sum_x r_{\pda}(x) \right)^2}{\sum_{x}  r_{\pda}(x)^2} = \frac{|A|^8}{D_\times (A)} > \frac{q}{2}. \qed
\]



\section{Proof of Theorem F}
\label{sec:thmE}

We prove a generalisation of~\cite[Theorem~27]{GPSumProd}. 
\begin{repthm}{thm:E}
For all constants $L>0$ there exists a constant $c>0$ such that for all primes $p$ and all sets $A,B,C,D \subseteq \mathbb{F}_p$ that satisfy $|A| |B| |C| |D| > L p^{12/5}$ and (assuming $|A| \leq |B|$, $|C| \leq |D|$, and $|B| \leq |D|)$ $|A| |B|^4 |C| \geq p^2 |D|^2$ , we have
\[
|(A-B)(C-D)| > c p.
\]
\end{repthm}
The result can be applied to sets like $(A+A^2)(A-A^{-1})$.

We use the following generalisation of the quantity $D_\times(A)$. Given four sets $A,B,C,D \subseteq \F_p$ we denote by $D_\times(A,B,C,D)$ the number of ordered solutions to the equation
\[
(a_1-b_1)(c_1-d_1) = (a_2-b_2)(c_2-d_2)
\]
with $a_1,a_2 \in A$, $b_1,b_2 \in B$, $c_1,c_2 \in C$ and $d_1,d_2 \in D$. We continue to write $D_\times(A)$ in place of $D_\times(A,A,A,A)$.

The proof is broken down in three steps. We first relate $D_\times(A,B,C,D)$ to $D_\times(A)$, $D_\times(B)$, $D_\times(C)$, and $D_\times(D)$; then we use a recent result of Roche-Newton, Rudnev, Shkredov, and the authors on collinear triples to bound $D_\times(A)$~\cite{GPSumProd}; the proof is then completed by the standard second moment argument given in Section~\ref{sec:small D}. The middle step, where all the non-trivial work lies, depends on Rudnev's points-planes incidence bound~\cite{Rudnev} and Lemma~\ref{lem:BKT}.

Let $D_\times(A)^*$ denote the non-zero solutions to $D_\times(A)$.
\begin{lemma}
  \label{lem:dxs}
Let $q$ be a prime power and $A,B,C,D$ be sets in $\F_q$. Suppose $|A| \leq |B|$, $|C| \leq |D|$ and $|B| \leq |D|$.
\begin{align*}
D_\times(A,B,C,D) 
& \leq  (D_\times(A)^* \, D_\times(B)^* \, D_\times(C)^* \, D_\times(D)^*)^{1/4}\\ 
& \quad \quad + 4 |A|^{2} |C|^{2} |D|^{2}.
\end{align*}
\end{lemma}

\begin{proof}
We decompose $\dxs$ in two parts: $\dxs^0$, which is the number of zero solutions to the equation defining $\dxs$; and $\dxs^*$ which is the number of non-zero solutions to the equation defining $\dxs$.

We first bound $\dxs^0$: It is at most
\begin{align*}
|A|^2  |C|^2 |D|^2  + 2 |A|^2 |B| |C|^2 |D|  +  |A|^2 |B|^2 |C|^2 
& =|A|^2 |C|^2  (|B| + |D|)^2 \\
& \leq 4 |A|^2 |C|^2 |D|^2.
\end{align*}

To bound $\dxs^*$ we use an argument similar to that in the proof of Theorem~\ref{thm:B}. The first step is to show
\[
D_\times(A,B,C,D)^* \leq (D_\times(A,B,A,B)^*)^{1/2} (D_\times(C,D,C,D)^*)^{1/2}.
\]
Indeed
\begin{align*}
D_\times(A,B,C,D)^*
& = \sum_{\xi \in \F_q^*} r_{\frac{A-B}{A-B}}(\xi) \, r_{\frac{C-D}{C-D}}(\xi) \\
& \leq \left(\sum_{\xi \in \F_q^*} r_{\frac{A-B}{A-B}}(\xi)^2\right)^{1/2}  \left(\sum_{\xi \in \F_q^*} r_{\frac{C-D}{C-D}}(\xi)^2 \right)^{1/2} \\
& (D_\times(A,B,A,B)^*)^{1/2} (D_\times(C,D,C,D)^*)^{1/2}.
\end{align*}
We proceed by showing, say,
\[
D_\times(A,B,A,B)^* \leq (D_\times(A)^*)^{1/2} (D_\times(B)^*)^{1/2}.
\]
Indeed,
\begin{align*}
D_\times(A,B,A,B)^*
& = \sum_{\xi \in \F_q^*} r_{\frac{A-B}{A-B}}(\xi)^2 \\
& = \sum_{\xi \in \F_q^*} \sum_{x \in \F_q}  r_{A-\xi A}(x) \, r_{B-\xi B}(x) \\
& \leq \sum_{\xi \in \F_q^*} \left(\sum_{x \in \F_q}  r_{A-\xi A}(x)^2 \right)^{1/2} \left(\sum_{x \in \F_q}  r_{B-\xi B}(x)^2 \right)^{1/2} \\
& \leq \left(\sum_{\xi \in \F_q^*} \sum_{x \in \F_q}  r_{A-\xi A}(x)^2\right)^{1/2} \left( \sum_{\xi \in \F_q^*} \sum_{x \in \F_q}  r_{B-\xi B}(x)^2 \right)^{1/2} \\
& \leq \left(\sum_{\xi \in \F_q^*}  r_{A- A}(\xi)^2\right)^{1/2} \left( \sum_{\xi \in \F_q^*} r_{B- B}(\xi)^2 \right)^{1/2} \\
& = (D_\times(A)^*)^{1/2} (D_\times(B)^*)^{1/2}.
\end{align*}
Therefore
\[
\dxs^* \leq (D_\times(A)^*)^{1/4} (D_\times(B)^*)^{1/4} (D_\times(C)^*)^{1/4} (D_\times(D)^*)^{1/4},
\]
which finishes the proof of the lemma.
\end{proof}

We now state the upper bound on $D_\times(A)$ given in~\cite{GPSumProd}. This time the bound only holds in $\F_p$.
\begin{theorem}
\label{thm:pds}
Let $p$ be a prime and $A \subseteq \F_p$ be a set.
\[
D_\times(A) - \frac{|A|^8}{p} \ll p^{1/2} |A|^{11/2}.
\]
\end{theorem}   
\begin{proof}
Let 
\[
T(A) = |\{ (a_1, \dots, a_6) \in A^6 : (a_1-a_2) (a_3-a_4) = (a_1-a_5) (a_3-a_6) \}|.
\]
An application of Cauchy-Schwarz (see \cite[Lemma~2.7]{GPTri}) gives
\[
D_\times(A) \leq |A|^2 \, T(A).
\]
Theorem 10 (1) in~\cite{GPSumProd} states
\[
T(A) - \frac{|A|^6}{p} \ll p^{1/2} |A|^{7/2}.
\]
Combining the two inequalities proves the theorem.
\end{proof}

Proving Theorem~\ref{thm:E} is now straightforward.

\begin{proof}[Proof of Theorem~\ref{thm:E}]

$\pds$ is the support of $r_{\pds}$. Similarly to the second part of the proof of Theorem~\ref{thm:A} we have
\[
|\pds| \geq \frac{\left( \sum_x r_{\pds}(x) \right)^2}{\sum_{x}  r_{\pds}(x)^2} = \frac{|A|^2 |B|^2 |C|^2 |D|^2}{D_\times (A,B,C,D)}.
\]
By Lemma~\ref{lem:dxs} and Theorem~\ref{thm:pds} we have
\begin{align*}
D_\times (A,B,C,D)^4
& \ll D_\times(A)^* \, D_\times(B)^* D_\times(C)^* \, D_\times(D)^* + (|A| |B| |D|)^8 \\
& \ll \frac{(|A| |B| |C| |D|)^8}{p^4}+ p^2 (|A| |B| |C| |D|)^{11/2} + (|A| |C| |D|)^8.
\end{align*}
The first term dominates when $|A| |B| |C| |D| \gg p^{12/5}$. For sufficiently large $C$ we get that if $|A| |B| |C| |D| > C p^{12/5}$, then $D_\times (A,B,C,D) \ll (|A| |B| |C| |D|)^2 /p$, in which case $|\pds| \gg p$.
\end{proof}


\section{Theorem C over prime order fields}
\label{sec:more on C}

Over  prime finite fields there are quantitatively strong versions of the results of Bourgain~\cite{Bourgain2009} and Bourgain--Glibichuk \cite{Bourgain-Glibichuk2011}.
We briefly state two such results for prime fields, for reference and comparison to the results in this paper.
The first was shown to the authors by Rudnev.

\begin{theorem}[Rudnev]
  \label{thm:C Misha}
Let $p$ be a prime, $A \subseteq \F_p$, and $X \subseteq \F_p^*$. Suppose that $|A|^2 |X| \ll p^2$.
\begin{align*}
\sum_{\xi \in X}E(A, \xi A) 
& \ll E(A)^{1/2} \, (|A|^{3/2} |X|^{3/4} + |A| |X|).
\end{align*}
 \end{theorem}
A careful analysis shows that the above implies that if $|A| \geq p^{1/2}$, then
\[
D_\times(A) - \frac{|A|^8}{p} \ll p^{2/3} |A|^{10/3} E(A)^{2/3}.
\]
This is stronger than Theorem~\ref{thm:pds} when $E(A) \leq p^{-1/4} |A|^{13/4}$.

The second result is in~\cite{GPCANTExp}. It is stronger than Theorem~\ref{thm:C Misha} when $E(A)\geq |A|^3/|X|^{1/6}\gg |A|^{5/2}$.
\begin{theorem}
  \label{thm:C Brendan}
Let $p$ be a prime, $A \subseteq \F_p$, and $X \subseteq \F_p^*$. Suppose that $|A| \ll p^{2/3}$ and $ |X| \ll |A|^3$.
\begin{align*}
\sum_{\xi \in X}E(A, \xi A) 
& \ll |A|^3 |X|^{2/3}.
\end{align*}
 \end{theorem}

\appendix 
\section*{Appendix}
\section[Proof of Proposition 4]{Proof of Proposition~\ref{prop:2}}
\label{app:A}

\subsection{Two combinatorial lemmata}

Many of the quantitative improvements we offer to existing results in the literature are based on the following elementary lemma.
These results are well known, but we state them here for precision in their application.

\begin{lem}[Popularity principle]
  \label{lem:pop}
Fix a number $0<\lambda<1$, let $S$ be a finite set, and let $f$ be a function such that $f(x)\geq 0$ for all $x$ in $S$.
\begin{enumerate}
\item Suppose that
\[
\sum_{x\in S}f(x)\geq \mu.
\]
Let $P_\lambda = \{x\in S\colon f(x)\geq \lambda \mu/|S|\}$.
Then
\begin{equation}
  \label{eq:35}
  \sum_{x\in P_\lambda}f(x)\geq (1-\lambda)\mu.
\end{equation}
Further, if $f(x)\leq M$ for all $x$ in $S$, then $|P_\lambda|\geq (1-\lambda)\mu/M$.
\item Suppose that for all $x$ in $S$, we have $w(x)\geq 0$ and $f(x)\leq M$.
Let  $W=\sum_{x\in S}w(x)$. If
\[
\sum_{x\in S}f(x)w(x)\geq \mu,
\]
then there exists a subset $P_*\subseteq S$ and a number $N$ such that
\begin{align}
  \label{eq:36}
  N&\leq f(x) < 2N\qquad\forall x\in P_* \nonumber\\
\frac{\lambda \mu}W&\leq N \leq M\\
 \frac{(1-\lambda)\mu}{\log_2M} & \leq \sum_{x\in P_*}f(x)w(x) \leq 2N\sum_{x\in P_*}w(x)\leq 2 N W \nonumber.
\end{align}
\end{enumerate}
\end{lem}
When $w(x)=1$ for all $x\in S$, the second part of this lemma is just a dyadic pigeonholing argument applied after the first part of lemma.
The general statement of the second part follows from dyadic pigeonholing applied after a weighted version of the first part.

We also use a standard combinatorial lemma.

\begin{lem}[Cauchy-Schwarz intersection lemma]
\label{lem:CSint}
 Let \(S\) be a finite index set and let \(T_s\) be a family of subsets of
a set \(T\).
Then
\begin{equation*}%
  \left(\sum_{s\in S}|T_s| \right)^2\leq |T|\sum_{s,s'\in S}|T_s\cap T_{s'}|.
\end{equation*}%
Further, if there exists \(\delta > 0\) such that
\[
 \sum_{s\in S}|T_s| \geq \delta |S||T|,
\]
then there exists a subset \(P\subseteq S\times S\) such that
\begin{enumerate}
\item \( |T_s\cap T_{s'}|\geq \delta^2 |T| / 2 \) for all pairs \((s,s')\)
   in \(P\).
\item \( |P|\geq \delta^2 |S|^2 / 2 \).
\end{enumerate}
\end{lem}


\subsection{The \bsg{} Theorem}
\label{sec:bsg}

Before proving Proposition~\ref{prop:2} we present a corollary of the \bsg{} Theorem~\cite{Balog-Szemeredi1994,Gowers1998} that suits our purposes.
Proposition~\ref{prop:1}, below, improves a result of Bourgain and Garaev~\cite[Lemmata 2.3 and 2.4]{bourgain2009variant}.

We follow Sudakov, Szemer\'edi, and Vu's graph-theoretic method~\cite{sudakov2005question}, using a slight improvement due to Tao and Vu~\cite[Corollary 6.20]{tao2010additive}.
\begin{lem}[Paths of length three]
\label{lem:1}
Let $K$ be a positive number, and let $G=G(A,B,E)$ be a bipartite graph with $|B|\leq |A|$ and $|E|=|A||B|/K$.
There exist subsets $A'\subseteq A$ and $B'\subseteq B$ such that
\[
|A'|\geq\frac{|A|}{4\sqrt{2}K}\andd |B'|\geq\frac{|B|}{4K},
\]
and for each $a\in A', b\in B'$ there are $|A||B|/(2^{12}K^5)$ paths of length 3 with endpoints $a$ and $b$.
\end{lem}
Note that there is a typo in the print version of \cite{tao2010additive}---the factor $K^4$ should be $K^5$; the errata can be found on the first author's webpage.

We combine Lemma~\ref{lem:1} with Lemma~\ref{lem:pop} to prove the following.
\begin{prop}
  \label{prop:1}
  Let $K$ be a positive number, and let $A$ and $B$ be finite subsets of a commutative group such that $|B|\leq |A|$ and $E(A,B)=(|A||B|)^{3/2}/K$.
  
Then there exist subsets $A'\subseteq A$ and $B'\subseteq B$ such that
\[
|A'|\gtrsim \frac{|A|}{K}\andd |B'|\gtrsim\frac {|B|}{K},
\]
and
\[
|A'+B'| \lesssim K^3(|A||B|)^{1/2}\lesssim K^4(|A'||B'|)^{1/2}.
\]
\end{prop}
Schoen \cite{schoen2015bounds} proved stronger results for the symmetric case of $E(A)\geq |A|^3/K$. 
\begin{proof}
  For short, let $E=E(A,B)$ and $L=\log_2|B|$.
Since
\[
E=E(A,B)=\sum_{x}r_{A+B}^2(x),
\]
by the popularity principle (\eqref{eq:36} in Lemma~\ref{lem:pop} with $\lambda= 1/2$, $S = A \times B$, $f(x) = w(x) = r_{A+B}(x)$, $M = |B|$, $W = |A| |B|$, $\mu = E$), there exists a number $N$ satisfying
\begin{equation}
  \label{eq:7}
\frac E{2|A||B|}\leq N \leq |B|
\end{equation}
such that the subset $G\subseteq A\times B$ defined by
\begin{equation}
  \label{eq:1}
G = \{(a,b)\in A\times B\colon N\leq r_{A+B}(a+b) < 2N\}
\end{equation}
satisfies
\begin{equation}
  \label{eq:3}
  \sum_{(a,b)\in G}r_{A+B}^2(a+b) = \sum_{x\in A\psum G B}r_{A+B}^2(x) \geq\frac E{2L}.
\end{equation}
Above, we denoted by $x\in A\psum G B$ the statement $x = a+b$ with $(a,b) \in G$. By \eqref{eq:1} and \eqref{eq:3} we have
\begin{equation}
  \label{eq:5}
  \frac E{4L} \leq N|G| 
\end{equation}
and
\begin{equation}
  \label{eq:6}
 N |A\psum G B| \leq |G|.
\end{equation}

Now we apply Lemma~\ref{lem:1} with $E=G$, and $K_0=|A||B|/|G|$  to find subsets $A'\subseteq A,B'\subseteq B$ such that
\[
|A'|\geq\frac{|A|}{4\sqrt{2}K_0}\andd |B'|\geq\frac{|B|}{4K_0},
\]
and each pair $a\in A'$ and $b\in B'$ is connected by at least $|A||B|/(2^{12}K_0^5)$ paths of length 3.

If $(a,b',a',b)$ is such a path, then we have
\[
a+b = (a+b')-(a'+b')+(a'+b) = x - x' + x''
\]
with $x,x',x''$ in $A\psum G B$.
Double counting the number of solutions to the above equation yields
\begin{equation*}%
 \frac{|A||B|}{2^{12}K_0^5}|A'+B'|\leq |A\psum G B|^3.
\end{equation*}%
By \eqref{eq:6} we have $|A\psum G B|\leq |G|/N$.
Since $K_0=|A||B|/|G|$, we get
\begin{equation*}%
  |A'+B'|\leq  \frac{2^{12}K_0^5}{|A||B|}|A\psum G B|^3 \leq \frac{2^{12}|A|^4|B|^4}{|G|^2N^3}.
\end{equation*}%
By \eqref{eq:5}, $1/(|G|^2N^2)\leq (4L/E)^2$ and by \eqref{eq:7}, $1/N \leq 2|A||B|/E$, so
\begin{equation*}%
  |A'+B'| \leq 2^{12}|A|^4|B|^4 \frac{2^4L^2}{E^2}\cdot\frac{2|A||B|}{E} = \frac{2^{17}|A|^5|B|^5L^2}{E^3}=2^{17}K^3L^2(|A||B|)^{1/2}.
\end{equation*}%
This proves the last claim. 

Finally, equations \eqref{eq:5}~and~\eqref{eq:7} imply that
\[
|G| \geq \frac E{4LN}\geq\frac E{4L|B|} = \frac{|A|^{3/2} |B|^{3/2}}{4LK|B|} \geq \frac{|A| |B|}{4LK},
\]
which yields $K_0=|A| |B| /|G|\leq 4LK$.
Hence
\[
|A'|\geq\frac{|A|}{4\sqrt{2}K_0}\geq\frac{|A|}{16\sqrt{2}LK}\andd |B'|\geq\frac{|B|}{4K_0}\geq \frac {|B|}{2^4LK}.
\]
\end{proof}


\subsection[Proof of Proposition 4]{Proof of Proposition~\ref{prop:2}}
\label{sec:bou-gli}

We now prove Proposition~\ref{prop:2}, which we restate here. The argument is essentially due to Bourgain~\cite[Proof of Theorem~C]{Bourgain2009}. Quantitative aspects were worked out by Bourgain and Glibichuk~{\cite[Proof of Proposition~2]{Bourgain-Glibichuk2011}}. We offer further quantitative improvements by a more careful analysis. 
\begin{repprop}{prop:2}
Let $A,X\subseteq\F_q^*$.
If there is a positive real number $K>0$ such that
\[
 \sum_{\xi\in X}E(A,\xi A) \geq \frac{|A|^3 |X|}K,
\]
then there exist elements $\bar a\in A$ and $x_0\in X$ and subsets $A_1 \subseteq A - \bar{a} $ and $X'\subseteq x_0^{-1} X$ such that:
\begin{enumerate}
\item[(i)] $\ds |A_1|\gtrsim K^{-85}|A|$ and $|X'|\gtrsim K^{-5}|X|$.
\item[(ii)] $ |A_1+X'A_1|\lesssim K^{226} |A_1|$.
\item[(iii)] $\ds |A_1 \pm A_1| \lesssim K^{92} |A_1|$.
\item[(iv)] $\ds |A_1 \pm x' A_1| \lesssim K^{126} |A_1|$ for all $x' \in X'$.
\end{enumerate}
Moreover, if $E(A,\xi A) \geq |A|^3 /K$ for all $\xi \in X$, then $|X'| \gtrsim K^{-4} |X|$.
\end{repprop}

As a comparison we state what can be read off of the proof of \cite[Proposition~2]{Bourgain-Glibichuk2011}: $|A_1| \gg K^{-441} |A|$, $|X'| \gg K^{-8} |X|$, and $|A_1+A_1X'| \ll K^{1490} |A_1|$.

The proof of Proposition~\ref{prop:2} follows from two lemmata.

\begin{lem} 
  \label{lem:3}
Let $A$ and $B$ be subsets of $\F_q^*$ and let $K$ be a positive constant.
Suppose that
\begin{equation*}%
  \sum_{b\in B}E(A,bA) \geq \frac{|A|^3|B|}K.
\end{equation*}%
There exist subsets $A'\subseteq A$ and $B'\subseteq b_0^{-1}B$ for some $b_0$ in $B$ such that 
\[
|A'\pm bA'|\lesssim K^{42} |A'|\quad\forall b\in B',
\]
\[
|A'\pm A'|\lesssim K^8|A'|,
\]
\[
|A'|\gtrsim\frac{|A|}K\andd |B'|\gtrsim\frac{|B|}{K^5}.
\]
Moreover, if $E(A,bA) \geq |A|^3 /K$ for all $b \in B$, then we have the same conclusion with $|B'| \gtrsim |B| / K^4$. 
\end{lem} 

\begin{lem}
  \label{lem:4}
Let $A$ and $B$ be subsets of $\F_q^*$ and let $K, K_0$ be a positive constants.
Suppose that $|A\pm bA|\leq K|A|$ for all $b$ in $B$, and that $|A\pm A|\leq K_0|A|$
.
There exists an element $\bar a$ in $A$ and a subset $A_1\subseteq A-\bar a$ such that
\[
|A_1|\gtrsim\frac{|A|}{K^2}
\andd |A\pm BA_1|\lesssim K^5K_0^2|A_1|.
\]
\end{lem}

\begin{proof}[Proof of Proposition~\ref{prop:2}]
\[
\sum_{\xi\in X}E(A,\xi A) \geq \frac{|A|^3|X|}K.
\]
By Lemma~\ref{lem:3} there exist subsets $A'\subseteq A$ and $X'\subseteq x_0^{-1}X$ such that
\[
|A'\pm \xi A'|\lesssim K^{42} |A'|\quad\forall \xi\in X',
\]
\[
|A'\pm A'|\lesssim K^8 |A'|,
\]
\[
|A'|\gtrsim K^{-1}|A|\andd |X'|\gtrsim K^{-5}|X|.
\]
The second part of (i) has been established.

We now apply Lemma~\ref{lem:4} to $A=A'$, $B=X'$, $K \lesssim K^{42}$ and $K_0 \lesssim K^8$. There is an element $\bar a$ in $A'$ and a subset $A_1\subseteq A'-\bar a$ such that
\[
|A_1|\gg K^{- 2 \cdot 42} |A'| = K^{-84}|A'| \gtrsim K^{-85}|A|
\]
and
\[
|A_1\pm X'A_1|\leq |A'\pm X'A_1| \ll K^{5 \cdot 42}K^{2 \cdot 8}|A_1| \lesssim K^{226}|A_1|.
\]
This proves (ii) and the first part of (i).

For (iii) note
\[
|A_1 \pm A_1| \leq |A' \pm A'| \lesssim K^{8} |A'| \lesssim K^{92} |A_1|.
\]
Similarly, for (iv) note
\[
|A_1 \pm x' A_1| \leq |A' \pm x' A'| \lesssim K^{42} |A'| \lesssim K^{126} |A_1|.
\]
The proof of the special case where $E(A, \xi A) \geq |A|^3 /K$ for all $\xi \in X$ is similar, the only difference is in the relative density of $X'$.
\end{proof}

\subsection[Proof of Lemma 22]{Proof of Lemma~\ref{lem:3}}

\begin{proof}[Proof of Lemma~\ref{lem:3}]
We are given
\[
\sum_{b\in B}E(A,bA) \geq \frac{|A|^3|B|}K.
\]
By the popularity principle (\eqref{eq:35} in Lemma~\ref{lem:pop} with $\lambda =1/2$, $S = B$, $f(b) = E(A,bA)$, $\mu = |A|^3 |B| /K$) there is a subset $B_1\subseteq B$ such that
\[
E(A,bA)\geq\frac{|A|^3}{2K}\qquad\forall b\in B_1\andd |B_1|\geq\frac{|B|}{2K}.
\]

\newcommand{\ab}{\ensuremath{A^{(b)}}}%
\newcommand{\abb}[1]{\ensuremath{A^{(#1)}}}%
For each $b$ in $B_1$, we apply Proposition~\ref{prop:1} with $A=A, B=bA$ and $K=2K$ to find $\ab_1,\ab_2\subseteq A$ such that
\begin{equation}
\label{eq:BG24}
|\ab_1+b\ab_2|\lesssim K^3|A|\andd |\ab_1|,|\ab_2|\gtrsim \frac{|A|}{K}.
\end{equation}

Now we find a large common subset of the $\ab_i$ by applying Lemma~\ref{lem:CSint} to the sets $\ab_1\times\ab_2\subseteq A\times A$ (we take $T = A \times A$, $S = B_1$, and $T_s = A_1^{(s)} \times A_2^{(s)}$).
Since 
\[
\sum_{b\in B_1}|\ab_1\times \ab_2| \gtrsim K^{-2}|A|^2|B_1|,
\]
there is a subset $P\subseteq B_1\times B_1$ such that
\begin{equation*}%
  |(\ab_1 \cap \abb{b'}_1)\times (\ab_2 \cap \abb{b'}_2)|\gtrsim K^{-4}|A|^2
\end{equation*}%
for all pairs $(b,b')$ in $P$, and $|P|\gtrsim K^{-4}|B_1|^2$.

Pigeonholing over $b'$ yields an element $b^*$ in $B_1$ such that
\begin{equation}
  \label{eq:17}
  |(\ab_1 \cap \abb{b^*}_1)\times (\ab_2 \cap \abb{b^*}_2)|\gtrsim K^{-4}|A|^2
\end{equation}
for all $b$ in $B_2\subseteq B_1$ and
\begin{equation}
  \label{eq:18}
  |B_2| \gtrsim K^{-4}|B_1| \gtrsim K^{-5}|B|.
\end{equation}
We will write $A_i^*$ for $\abb{b^*}_i$.

Next we apply~\eqref{eq:BG24} and the sumset cardinality inequalities in Lemma~\ref{pivot-lem:PR} to show that
\begin{equation}
  \label{eq:19}
  |A_2^*\pm A_2^*| \lesssim K^7|A|
\end{equation}
and for all $b \in B_2$
\begin{equation}
  \label{eq:20}
  |b^*A^*_2\pm bA^*_2|\lesssim K^{41}|A|\quad\forall b\in B_2.
\end{equation}
First we prove a stronger version of \eqref{eq:19}: Let $b \in B_1$
\begin{equation}
  \label{eq:21}
    |A_2^{(b)}+A_2^{(b)}|=  |bA_2^{(b)}+bA_2^{(b)}|\leq \frac{|A_1^{(b)}+bA_2^{(b)}|^2}{|A_1^{(b)}|} \overset{\eqref{eq:BG24}}\lesssim K^6 \frac{|A|^2}{|A_1^{(b)}|} \overset{\eqref{eq:BG24}}\lesssim K^{7}|A|.
\end{equation}
This proves~\eqref{eq:19} by taking $b=b^*$, which is an element of $B_1$. Similarly, for all $b \in B_1$,
\begin{equation}
  \label{eq:22}
    |A_1^{(b)}\pm A_1^{(b)}| \leq \frac{|A_1^{(b)}+bA_2^{(b)}|^2}{|A_2^{(b)}|} \lesssim K^{7}|A|.
\end{equation}

Now we prove \eqref{eq:20}.
\begin{align*}
  |b^* A^{*}_2\pm b A^{*}_2| 
&\leq \frac{|A^*_1+b^*A_2^*||A_1^*+bA_2^*|}{|A_1^*|}\\
&\overset{\eqref{eq:BG24}}\lesssim K^{4}|A_1^*+bA_2^*|\\
&\lesssim K^{4}\frac{|A_1^*+A_1^{(b)}||A_1^{(b)}+bA_2^*|}{|A_1^{(b)}|}\\
&\lesssim K^{4}\frac{|A_1^*+A_1^{(b)}||A_1^{(b)}+bA_2^{(b)}||A_2^{(b)}+A_2^*|}{|A_1^{(b)}||A_2^{(b)}|}\\
&\overset{\eqref{eq:BG24}}\lesssim K^{9}\frac{|A_1^*+A_1^{(b)}||A_2^*+A_2^{(b)}|}{|A|}.
\end{align*}
x
We use \eqref{eq:BG24}, \eqref{eq:17}, \eqref{eq:21} and \eqref{eq:22} to estimate the remaining sumsets:
\begin{align*}
  |A_1^*+A_1^{(b)}||A_2^*+A_2^{(b)}| 
&\leq \frac{|A_1^*+A_1^*||A_1^{(b)}+A_1^{(b)}||A_2^*+A_2^*||A_2^{(b)}+A_2^{(b)}|}{|A_1^*\cap A_1^{(b)}||A_2^*\cap A_2^{(b)}|}\\
&\lesssim \frac{(K^{7}|A|)^4}{K^{-4}|A|^2}\lesssim K^{32}|A|^2.
\end{align*}
Combined with the previous chain of inequalities, this proves \eqref{eq:20}.

Setting $A'=A_2^*$ and $B'=(b^*)^{-1}B_2$, we have by \eqref{eq:BG24} and \eqref{eq:20}
\[
|A'\pm b(b^*)^{-1} A'| = |b^*A_2^*\pm bA_2^*| \lesssim K^{41}|A| \lesssim K^{42}|A'|
\]
for all $b(b^*)^{-1}$ in $B'$; and by \eqref{eq:19}
\[
|A'\pm A'| \lesssim K^7|A| \lesssim K^8|A'|.
\]
Setting $b_0=b^*$ and invoking \eqref{eq:18} completes the proof of the most general statement.

If $E(A,bA) \geq |A|^3/K$ for all $b \in B$, then we may take $B_1=B$ and so the lower bound on $|B_2|$ in \eqref{eq:18} becomes $|B|/K^{4}$.
\end{proof}

\subsection[Proof of Lemma 23]{Proof of Lemma~\ref{lem:4}}

\begin{proof}[Proof of Lemma~\ref{lem:4}]
We begin with a simple observation we use further down: we may without loss of generality assume that $|B| \geq 32 K_0 K^3$ because otherwise the conclusion follows immediately with $A_1 = A $:
\[
|A \pm B A| \leq \sum_{b \in B} |A \pm b A| \leq K |B| |A| \ll K^4 K_0 |A|.
\]

For a number $\tau\geq |B|/(2K^2)$ we say that a non-zero $\xi\neq 0$ is \emph{$\tau$-involved}\/ if
  \begin{equation}
    \label{eq:bou-gli11}
   \min_{\substack{B'\subseteq B, |B'|\geq\tau \\ Y\subseteq A, |Y|\geq |A|/2}} |Y+B'\xi| \leq 4K_0K|A|.
  \end{equation}
If $|Y+B'\xi|\leq 4K_0K|A|$ for some $B'\subseteq B$ and $Y\subseteq A$ satisfying $|B'|\geq\tau, |Y|\geq |A|/2$, we say that $\xi$ is \emph{witnessed} by $B'$.

Let $\Omega$ denote the set of $\tau$-involved elements. We obtain an upper bound on $\Omega$. First observe that for any subsets $A'\subseteq A, B' \subseteq B$ we have
\[
E(A,\xi B) -|A||B|\geq E(A', \xi B')-|A'||B'|,
\]
because, say, $E(A,\xi B) -|A||B|$ is the number of non-zero ordered solutions to $(a-a') = \xi (b-b')$ with $a,a' \in A$ and $b,b' \in B$. For any involved $\xi$, we have
\begin{align*}
E(A,\xi B) -|A||B| 
& \geq E(Y, \xi B')-|Y||B'|  \\
& \geq \frac{|Y|^2|B'|^2}{|Y+B'\xi|} -|Y||B'| \\
& \geq \frac{|A||B'|}2\left(\frac{|B'|}{8K_0K}-1 \right).
\end{align*}
We may assume $|B| \geq 32 K_0 K^3$, as detailed at the beginning of the proof, and so we have $|B'| \geq |B|/(2K^2) \geq 16 K_0 K$.
Thus $|B'|/(8K_0 K)-1 \geq |B'| /(16K_0 K)$ and we get for all involved $\xi$
\[
E(A,\xi B) - |A||B| \geq \frac{|A| |B'|^2}{32 K_0K} \geq \frac{|A|\tau^2}{32K_0K}.
\]
Hence we get the following upper bound on $\Omega$ via Lemma~\ref{lem:BKT}
\begin{equation}
  \label{eq:bou-gli12}
|\Omega| \leq \frac{32 KK_0}{|A|\tau^2}\sum_{\xi}(E(A,\xi B)-|A||B|) \leq \frac{32 K K_0}{|A|\tau^2}|A|^2|B|^2 \leq \frac{64 K^3 K_0|A||B|}{\tau}.
\end{equation}

We use this upper bound on $|\Omega|$ further down. Our next task is to find elements in $\Omega$. To do so, we must find particular $\xi$, $B'$, and $Y$ for which \eqref{eq:bou-gli11} holds. Our aim is to identify two $\xi$ for each $b \in B$. We begin with an averaging argument. Since $|A+bA|\leq K|A|$ for all $b$ in $B$, we have
\[
\sum_{b\in B}E(A,bA)\geq \frac{|A|^3|B|}K.
\]
By pigeonholing, there exist $\bar a, \Bar{\Bar a}\in A$ such that
\begin{equation*}%
  |\{(a_1,a_2,b)\in A^2\times B\colon a_2-\Bar{\Bar a} = b(a_1 -\bar a)\}| \geq\frac{|A||B|}K.
\end{equation*}%
For each $a$ in $A$, let
\[
B_a=\{b\in B\colon b(a-\bar a)\in A-\Bar{\Bar a}\}.
\]
In this notation we have
\[
\sum_{a\in A}|B_a|\geq \frac{|A||B|}K,
\]
and so, by the first part of Lemma~\ref{lem:CSint},
\[
\sum_{(a,a')\in A\times A}|B_a\cap B_{a'}| \geq\frac{|A|^2|B|}{K^2}.
\]
Thus by the popularity principle (\eqref{eq:36} in Lemma~\ref{lem:pop} with $S=A\times A$, $f(a,a')=|B_a\cap B_{a'}|$, $w(a,a')=1, W = |A|^2, M=|B|$, $\mu = |A|^2 |B|/K^2$, $\lambda=1/2$) there exists a subset $P_*\subseteq A\times A$ and a number $|B|/(2K^2)\leq \tau\leq |B|$ such that
\begin{equation}
  \label{eq:bou-gli14}
\tau\leq  |B_a\cap B_{a'}| < 2\tau
\end{equation}
for all $(a,a')\in P_*$ and
\begin{equation}
  \label{eq:bou-gli15}
\frac{|A|^2 |B|}{2K^2\log_2|B|}  \leq\sum_{(a,a')\in P_*} |B_a\cap B_{a'}|\leq   2\tau|P_*|.
\end{equation}

Now fix $b \in B$. For each pair $(a,a')$ in $P_*$, we define
\[
\xi_\pm = (a'-\bar a)\pm b(a-\bar a).
\]
We show that both elements $\xi_\pm$ belong to $\Omega$ because they are witnessed by the set $B'=B_a\cap B_{a'}$, which, by~\eqref{eq:bou-gli14} has cardinality at least $\tau\geq|B|/(2K^2)$.
To see why $\xi_\pm$ are witnessed by $B'$, recall that by its definition, $B_a$ is the set of $b$ such that $b(a-\bar a) \in A - \bar{\bar a}$.
Therefore 
for any set $Y$ we have
\[
  |Y+B'\xi_\pm|\leq |Y+B_{a'}(a'-\bar a)\pm b B_a(a-\bar a)|
\leq |Y+A\pm bA|.
\]
By Pl\"unnecke's inequality for different summands with a large subset in Lemma~\ref{pivot-lem:PR}, there exists a subset $Y\subseteq A$ such that $|Y|\geq |A|/2$ and
\begin{align*}
|Y+A\pm bA|\leq 4\,\frac{|Y+A|}{|A|} \,  \frac{|Y \pm bA|}{|A|} \; |A| \leq 4K_0 K |A|, 
\end{align*}
which confirms \eqref{eq:bou-gli11}.

This proves that $\xi_\pm \in \Omega$. This is true for all $b \in B$  and therefore
\begin{equation}
  \label{eq:bou-gli16}
  (a'-\bar a) \pm B(a-\bar a) \subseteq \Omega\quad\forall (a,a')\in P.
\end{equation}

The final step of the proof is to get an upper bound on $|A+A_1B|$ for some suitable set $A_1$ (of relative density $\gtrsim 1/K^2$ in $A$) by double-counting the number of representations of elements of $A+A_1B$ as a sum of elements in $A-A_1$ and $\Omega$.

Let $P(a)=\{a'\in A\colon (a,a')\in P_*\}$, so that by~\eqref{eq:bou-gli15}
\[
|P_*| = \sum_{a\in A}|P(a)|  \geq\frac{|A|^2|B|}{4\tau K^2\log_2|B|}.
\]
By the popularity principle (\eqref{eq:36} in Lemma~\ref{lem:pop} with $\lambda = 1/2, S=A, f(a)=|P(a)|$, $w(a)=1, W = |A|, M=|A|$, and $\mu = |P_*|$) there is a subset $A_*\subseteq A$ and an integer $|A|/4K^2\leq N\leq |A|$ such that $N\leq |P(a)| < 2N$ for all $a$ in $A_*$ and
\begin{equation}
  \label{eq:34}
 \frac{|A|^2|B|}{\tau K^2} \lesssim |P_*| \lesssim N |A_*|.
\end{equation}
Note that
\begin{equation}
  \label{eq:bou-gli18}
  |A_*|\gtrsim \frac{|A|}{K^2}
\end{equation}
since $N\leq |A|$ and $\tau\leq |B|$.

Set $A_1=A_*-\bar a$ and  
$r(x) :=r_{(A-A+ \bar a)+\Omega}(x)$.
We show that for all $x$ in $A+A_1B$
\begin{equation}
  \label{eq:bou-gli17}
r(x) \geq N  \quad \forall x \in A+A_1B.
\end{equation}
To see why consider $x=\tilde a \pm (a-\bar a)b \in A \pm A_1B$ (so $a \in A_*$). For all $a'\in P(a)$, inclusion \eqref{eq:bou-gli16} gives
\[
  x = \tilde a \pm (a-\bar a)b = [\tilde a - (a'-\bar a)] + [(a'-\bar a) \pm (a-\bar a)b]\in (A-A + \bar a)+\Omega. 
\]
Distinct $a' \in P(a)$ give distinct representations and so we indeed get $r(x) \geq |P(a)| \ge N$.

Now we double count:
\begin{align*}
N|A\pm A_1B|
 & \overset{\eqref{eq:bou-gli17}} \leq\sum_{x\in (A-A+\bar a)+\Omega}r(x) \\ 
&\leq |A-A||\Omega| \\
& \overset{\eqref{eq:bou-gli12}}\leq (K_0|A|)(2^6 K^3 K_0|A||B|\tau^{-1}).
\end{align*}

We are ready to conclude the proof. By \eqref{eq:bou-gli18}, $|A_1| = |A_*| \gtrsim |A|/K^2$ and by the above inequality and \eqref{eq:34}
\[
|A \pm A_1B| \ll \frac{K^3K_0^2 |A|^2|B|}{\tau N}
 \lesssim \frac{K^3K_0^2|A|^2|B||A_*|}{\tau|P_*|}
\lesssim 
 K^5K_0^2|A_1|. 
\]
\end{proof}

\section*{Acknowledgments} 
The authors are grateful to the anonymous reviewers, as well as Brandon Hanson, Alex Iosevich, Olly Roche-Newton, Misha Rudnev, and Igor Shparlinski.

\phantomsection

\addcontentsline{toc}{section}{References}

\bibliographystyle{amsplain}

\begin{thebibliography}{99}
\bibitem{Alon-Chung1988}
N.~Alon and {F. K. R.} Chung.
\newblock Explicit construction of linear sized tolerant networks.
\newblock {\em Discrete Math.}, 72(1):15--19, 1988.

\bibitem{Balog-Szemeredi1994}
A.~Balog and E.~Szemer{\'e}di.
\newblock A statistical theorem of set addition.
\newblock {\em Combinatorica}, 14:585--613, 1994.

\bibitem{BHIPR2017}
M.~Bennett, D.~Hart, A.~Iosevich, J.~Pakianathan, and M.~Rudnev.
\newblock Group actions and geometric combinatorics in {$\Bbb{F}_q^d$}.
\newblock {\em Forum Math.}, 29(1):91--110, 2017.

\bibitem{Bourgain2009}
J.~Bourgain.
\newblock Multilinear exponential sums in prime fields under optimal entropy
  condition on the sources.
\newblock {\em Geom. Funct. Anal.}, 18:1477--1502, 2009.

\bibitem{bourgain2009variant}
J.~Bourgain and {M. Z.} Garaev.
\newblock On a variant of sum-product estimates and explicit exponential sum
  bounds in prime fields.
\newblock {\em Math. Proc. Cambridge Philos. Soc.}, 146(1):1--21, 2009.

\bibitem{Bourgain-Glibichuk2011}
J.~Bourgain and {A.} Glibichuk.
\newblock Exponential sum estimates over a subgroup in an arbitrary finite
  field.
\newblock {\em J. Anal. Math.}, 115(1):51--70, 2011.

\bibitem{BGK2006}
J.~Bourgain, {A. A.} Glibichuk, and {S. V.} Konyagin.
\newblock Estimates for the number of sums and products and for exponential
  sums in fields of prime order.
\newblock {\em J. London Math. Soc. (2)}, 73:380--398, 2006.

\bibitem{BKT2004}
J.~Bourgain, {N. H.} Katz, and T.~Tao.
\newblock A sum-product estimate in finite fields, and applications.
\newblock {\em Geom. Funct. Anal.}, 14(1):27--57, 2004.

\bibitem{Bukh-Tsimerman2012}
B.~Bukh and J.~Tsimerman.
\newblock Sum-product estimates for rational functions.
\newblock {\em Proc. Lond. Math. Soc. (3)}, 104(1):1--26, 2012.

\bibitem{CEHIK2012}
J.~Chapman, {M. B.} Erdo{\u{g}}an, D.~Hart, A.~Iosevich, and D.~Koh.
\newblock Pinned distance sets, {$k$}-simplices, {W}olff's exponent in finite
  fields and sum-product estimates.
\newblock {\em Math. Z.}, 271(1-2):63--93, 2012.

\bibitem{Conrad2002}
K.~Conrad.
\newblock On {W}eil's proof of the bound for {K}loosterman sums.
\newblock {\em J. Number Theory}, 97(2):439--446, 2002.

\bibitem{Elekes1998}
G.~Elekes.
\newblock A combinatorial problem on polynomials.
\newblock {\em Discrete Comput. Geom.}, 19(3):383--389, 1998.

\bibitem{Elekes-Ronyai2000}
G.~Elekes and L.~R{\'o}nyai.
\newblock A combinatorial problem on polynomials and rational functions.
\newblock {\em J. Combin. Theory Ser. A}, 89(1):1--20, 2000.

\bibitem{Garaev2008}
{M. Z.} Garaev.
\newblock The sum-product estimate for large subsets of prime fields.
\newblock {\em Proc. Amer. Math. Soc.}, 136(8):2735--2739, 2008.

\bibitem{gill2014growth}
N.~Gill and {H. A.} Helfgott.
\newblock Growth in solvable subgroups of {${\rm GL}_r(\Bbb Z/p\Bbb Z)$}.
\newblock {\em Math. Ann.}, 360(1-2):157--208, 2014.

\bibitem{Glibichuk-Konyagin2007}
{A. A.} Glibichuk and {S. V.} Konyagin.
\newblock Additive properties of product sets in fields of prime order.
\newblock In A.~Granville, {M. B.} Nathanson, and J.~Solymosi, editors, {\em
  Additive Combinatorics, CRM Proceedings \& Lecture Notes 43}, pages 279--286,
  Providence, {R.I.}, 2007. Amer. Math. Soc.

\bibitem{Gowers1998}
{W. T.} Gowers.
\newblock A new proof of {S}zemer{\'{e}}di's theorem for arithmetic
  progressions of length four.
\newblock {\em Geom. Funct. Anal.}, 8(3):529--551, 1998.

\bibitem{Gowers2001}
{W. T.} Gowers.
\newblock A new proof of {S}zemer{\'{e}}di's theorem.
\newblock {\em Geom. Funct. Anal.}, 11(3):465--588, 2001.

\bibitem{Guth-Katz2015}
L.~Guth and {N. H.} Katz.
\newblock On the {E}rd{\H{o}}s distinct distances problem in the plane.
\newblock {\em Ann. of Math. (2)}, 181(1):155--190, 2015.

\bibitem{Gyarmati-Sarkozy2008}
K.~Gyarmati and {A.} S{\'{a}}rk{\"{o}}zy.
\newblock Equations in finite fields with restricted solution sets. {II}
  (algebraic equations).
\newblock {\em Acta Math. Hungar.}, 119(3):259--280, 2008.

\bibitem{Haemers1979}
{W. H.} Haemers.
\newblock {\em Eigenvalue techniques in design and graph theory}.
\newblock PhD thesis, Technische Hogeschool Eindhoven, 1979.

\bibitem{Hart-Iosevich2008}
D.~Hart and A.~Iosevich.
\newblock Sums and products in finite fields: an integral geometric viewpoint.
\newblock In {\em Radon {T}ransforms, {G}eometry, and {W}avelets}, AMS
  Contemporary Mathematics 464, pages 129--136. AMS RI, 2008.

\bibitem{HIS2007}
D.~Hart, A.~Iosevich, and J.~Solymosi.
\newblock Sum-product estimates in finite fields via {K}loosterman sums.
\newblock {\em Int. Math. Res. Not. IMRN}, 2007:Art. ID rnm007, 14, 2007.

\bibitem{HLS2013}
D.~Hart, L.~Li, and {C.-Y.} Shen.
\newblock Fourier analysis and expanding phenomena in finite fields.
\newblock {\em Proc. Amer. Math. Soc.}, 141(2), 2013.

\bibitem{Heath-Brown-Konyagin2000}
{D. R.}~Heath-Brown and {S. V.}~Konyagin.
\newblock New bounds for {G}auss sums derived from {$k$}th powers, and for
  {H}eilbronn's exponential sum.
\newblock {\em Q. J. Math.}, 52(2):221--235, 2000.

\bibitem{Hegyvari-Hennecart2013}
N.~Hegyv{\'a}ri and F.~Hennecart.
\newblock Conditional expanding bounds for two-variable functions over prime
  fields.
\newblock {\em European J. Combin.}, 34(8):1365--1382, 2013.

\bibitem{helfgott2008growth}
H.~A.~Helfgott.
\newblock Growth and generation in {${\rm SL}_2(\Bbb Z/p\Bbb Z)$}.
\newblock {\em Ann. of Math. (2)}, 167(2):601--623, 2008.

\bibitem{helfgott2011growth}
H.~A.~Helfgott.
\newblock Growth in {${\rm SL}_3(\Bbb Z/p\Bbb Z)$}.
\newblock {\em J. Eur. Math. Soc. (JEMS)}, 13(3):761--851, 2011.

\bibitem{helfgott2015growth}
{H. A.}~Helfgott.
\newblock Growth in groups: ideas and perspectives.
\newblock {\em Bull. Amer. Math. Soc. (N.S.)}, 52(3):357--413, 2015.

\bibitem{IRNR}
A.~Iosevich, O.~Roche-Newton, and M.~Rudnev.
\newblock On discrete values of bilinear forms.
\newblock {\em Mat. Sb.}, 209(10):71--88, 2018.

\bibitem{Iwaniec-Kowalski2004}
H.~Iwaniec and E.~Kowalski.
\newblock {\em Analytic Number Theory}.
\newblock Amer. Math. Soc., Providence, {R.I.}, 2004.

\bibitem{katz2008slight}
{N. H.} Katz and {C.-Y.} Shen.
\newblock A slight improvement to {G}araev's sum product estimate.
\newblock {\em Proc. Amer. Math. Soc.}, 136(7):2499--2504, 2008.

\bibitem{Konyagin2003}
S.~Konyagin.
\newblock A sum-product estimate in fields of prime order.
\newblock \href{http://arXiv:math/0304217}{arXiv:math/0304217}, 2003.

\bibitem{lang2002algebra}
S.~Lang.
\newblock {\em Algebra}, volume 211 of {\em Graduate Texts in Mathematics}.
\newblock Springer-Verlag, New York, third edition, 2002.

\bibitem{GPInc}
B.~Murphy and G.~Petridis.
\newblock A point-line incidence identity in finite fields, and applications.
\newblock {\em Mosc. J. Comb. Number Theory}, 6(1):63--94, 2016.

\bibitem{GPCANTExp}
B.~Murphy and G.~Petridis.
\newblock A second wave of expanders over finite fields.
\newblock In  Combinatorial and additive number theory. II, 215--238,
Springer Proc. Math. Stat., 220, Springer, Cham, 2017.

\bibitem{GPSumProd}
B.~Murphy, G.~Petridis, O.~Roche-Newton, M.~Rudnev, and {I. D.}~Shkredov.
\newblock New results on sum-product type growth over fields.
\newblock \href{http://arxiv.org/abs/1702.01003}{arXiv:1702.01003}, 2017.

\bibitem{MRS}
B.~Murphy, O.~Roche-Newton, and {I. D.}~Shkredov.
\newblock Variations on the sum-product problem {II}.
\newblock {\em SIAM J. Discrete Math.}, 31(3):1878--1894, 2017.

\bibitem{GPNAP}
G.~Petridis.
\newblock New proofs of {P}l{\"{u}}nnecke-type estimates for product sets in
  groups.
\newblock {\em Combinatorica}, 32(6):721--733, 2012.

\bibitem{GPProdDiff}
G.~Petridis.
\newblock Products of differences in prime order finite fields.
\newblock \href{http://arxiv.org/abs/1602.02142}{arXiv:1602.02142}, 2016.

\bibitem{GPTri}
G.~Petridis and {I. E.}~Shparlinski.
\newblock Bounds on trilinear and quadrilinear exponential sums.
\newblock Accepted in J. Anal. Math.
  \href{http://arxiv.org/abs/1604.08469}{arXiv:1604.08469}, 2016.

\bibitem{PVZ}
T.~Pham, {L. A.} Vinh, and {F. de} Zeeuw.
\newblock Three-variable expanding polynomials and higher-dimensional distinct
  distances.
\newblock {\em Combinatorica}.
\newblock To appear.

\bibitem{Plunnecke1970}
H.~Pl{\"{u}}nnecke.
\newblock Eine zahlentheoretische anwendung der graphtheorie.
\newblock {\em J. Reine Angew. Math.}, 243:171--183, 1970.

\bibitem{RSS2016}
{O. E.} Raz, M.~Sharir, and J.~Solymosi.
\newblock Polynomials vanishing on grids: the {E}lekes-{R}{\'o}nyai problem
  revisited.
\newblock {\em Amer. J. Math.}, 138(4):1029--1065, 2016.

\bibitem{RNR2015}
O.~Roche-Newton and M.~Rudnev.
\newblock On the minkowski distances and products of sum sets.
\newblock {\em Israel J. Math.}, 209:507--526, 2015.

\bibitem{Rudnev}
M.~Rudnev.
\newblock On the number of incidences between planes and points in three
  dimensions.
\newblock {\em Combinatorica}, 38(1):219--254, 2018.

\bibitem{RSS}
M.~Rudnev, {I. D.}~Shkredov, and S.~Stevens.
\newblock On the energy variant of the sum-product conjecture.
\newblock Preprint. \href{http://arxiv.org/abs/1607.05053}{arXiv:1607.05053}.

\bibitem{Ruzsa1978}
{I. Z.} Ruzsa.
\newblock On the cardinality of {${A}+{A}$} and {${A}-{A}$}.
\newblock In A.~Hajnal and V. T. S{\'{o}}s, editors, {\em Combinatorics
  (Keszthely 1976) Coll. Math. Soc. J. Bolyai, vol 18}, pages 933--938.
  North-Holland -- Bolyai T{\'a}rsulat, 1978.

\bibitem{Ruzsa1989}
{I. Z.} Ruzsa.
\newblock An application of graph theory to additive number theory.
\newblock {\em Scientia, Ser. A}, 3:97--109, 1989.

\bibitem{Ruzsa1992}
{I. Z.} Ruzsa.
\newblock Arithmetical progressions and the number of sums.
\newblock {\em Period. Math. Hungar.}, 25(1):105--111, 1992.

\bibitem{Ruzsa2009}
{I. Z.} Ruzsa.
\newblock Sumsets and structure.
\newblock In {\em Combinatorial Number Theory and Additive Group Theory, Adv.
  Courses Math. CRM Barcelona}, pages 87--210. Birkh{\"{a}}user Verlag, Basel,
  2009.

\bibitem{Sarkozy2005}
A.~S{\'{a}}rk{\"{o}}zy.
\newblock On sums and products of residues modulo {$p$}.
\newblock {\em Acta Arith.}, 118(4):40--409, 2005.

\bibitem{schoen2015bounds}
T.~Schoen.
\newblock New bounds in {B}alog-{S}zemer\'edi-{G}owers theorem.
\newblock {\em Combinatorica}, 35(6):695--701, 2015.

\bibitem{Shkredov2010}
{I. D.}~Shkredov.
\newblock On monochromatic solutions of some nonlinear equations in
  {$\mathbb{Z}/p\mathbb{Z}$}.
\newblock {\em Mathematical Notes}, 88(3):603--611, 2010.

\bibitem{Shkredov2016B}
{I. D.}~Shkredov.
\newblock On a question of {A}. {Balog}.
\newblock {\em Pacific J. Math.}, 280(1):227--240, 2016.

\bibitem{Shparlinski2008}
{I. E.} Shparlinski.
\newblock On the solvability of bilinear equations in finite fields.
\newblock {\em Glasg. Math. J.}, 50(3):523--529, 2008.

\bibitem{Solymosi2008}
J.~Solymosi.
\newblock Incidences and the spectra of graphs.
\newblock In M.~Gr{\"{o}}tschel, {G.O.H.} Katona, and G.~S{\'{a}}gi, editors,
  {\em Building Bridges: between mathematics and computer science, Bolyai
  Society Mathematical Studies Vol.~19}, pages 499--513. Springer, Berlin
  Heidelberg, 2008.

\bibitem{SdZ}
S.~Stevens and F.~de Zeeuw.
\newblock An improved point-line incidence bound over arbitrary fields.
\newblock {\em Bull. Lond. Math. Soc.}, 49(5):842--858, 2017.

\bibitem{sudakov2005question}
B.~Sudakov, E.~Szemer{\'e}di, and {V. H.} Vu.
\newblock On a question of {E}rd{\H{o}}s and {M}oser.
\newblock {\em Duke Math. J.}, 129(1):129--155, 2005.

\bibitem{Tao2015}
T.~Tao.
\newblock Expanding polynomials over finite fields of large characteristic, and
  a regularity lemma for definable sets.
\newblock {\em Contrib. Discrete Math.}, 10(1):22--98, 2015.

\bibitem{tao2015expansion}
T.~Tao.
\newblock {\em Expansion in finite simple groups of {L}ie type}, volume 164 of
  {\em Graduate Studies in Mathematics}.
\newblock American Mathematical Society, Providence, RI, 2015.

\bibitem{tao2010additive}
T.~Tao and V.~H. Vu.
\newblock {\em Additive combinatorics}, volume 105 of {\em Cambridge Studies in
  Advanced Mathematics}.
\newblock Cambridge University Press, Cambridge, 2010.

\bibitem{Vinh2011}
{L. A.} Vinh.
\newblock The {S}zemer{\'{e}}di-{T}rotter type theorem and the sum-product
  estimate in finite fields.
\newblock {\em European J. Combin.}, 32(8):1177--1181, 2011.

\bibitem{Vinh2013}
{L. A.} Vinh.
\newblock On four-variable expanders in finite fields.
\newblock {\em SIAM J. Discrete Math.}, 27(4):2038--2048, 2013.

\bibitem{Vu2008}
{V. H.} Vu.
\newblock Sum-product estimates via directed expanders.
\newblock {\em Math. Res. Lett.}, 15(2):375--378, 2008.

\bibitem{dZ}
{F. de} Zeeuw.
\newblock A short proof of {R}udnev's point-plane incidence bound.
\newblock \href{http://arxiv.org/abs/1612.02719}{arXiv:1612.02719}, 2016.

\end{thebibliography}


\begin{dajauthors}
\begin{authorinfo}[bm]
  Brendan Murphy\\
  School of Mathematics\\
  University of Bristol\\
  Bristol, United Kingdom\\
  brendan\imagedot{}murphy\imageat{}bristol\imagedot{}ac\imagedot{}uk \\
  \url{http://www.bristol.ac.uk/maths/people/brendan-m-murphy/overview.html}
\end{authorinfo}
\begin{authorinfo}[gp]
  Giorgis Petridis\\
  Department of Mathematics\\
  University of Georgia\\
  Athens, GA 30602\\
  USA\\
  giorgis\imageat{}cantab\imagedot{}net \\
  \url{https://faculty.franklin.uga.edu/petridis/}
\end{authorinfo}
\end{dajauthors}

\end{document}